%% file: Zeta_Distributions_Version_04_07_Ver2.tex
\documentclass[reqno]{amsart}
\usepackage{amsfonts, mathtools, amsmath, amsthm, amssymb, latexsym, xfrac, mathrsfs, braket}
\usepackage{framed,graphicx,xcolor,enumitem}
\definecolor{shadecolor}{gray}{0.9}

\theoremstyle{plain}
\newtheorem{theorem}{Theorem}[section]
\newtheorem{corollary}[theorem]{Corollary}
\newtheorem{lemma}[theorem]{Lemma}
\newtheorem{proposition}[theorem]{Proposition}
\theoremstyle{definition}
\newtheorem{definition}[theorem]{Definition}
\theoremstyle{remark}
\newtheorem{remark}[theorem]{Remark}

\newtheorem{example}[theorem]{Example}
\numberwithin{equation}{section}

\input{MyMacros}

\title[Hankel transform, $\mathcal{K}$-Bessel fucntions and Zeta distributions]
{Hankel transform, $\mathcal{K}$-Bessel functions and zeta distributions in the Dunkl setting}
\author{Dominik Brennecken} 
\address{Institut f\"ur Mathematik, Universit\"at Paderborn, Warburger Str. 100, D-33098 Paderborn, Germany}
\email{bdominik@math.upb.de}


\begin{document}
\date{\today}

\begin{abstract}
We study analytic properties of a Hankel transform for the type $A$ Dunkl setting with arbitrary multiplicity parameter $k\ge 0$, which goes back to Baker and Forrester and, in an earlier symmetrized version, to Macdonald. Moreover, we introduce a Dunkl analogue of the Bessel function and $\mathcal{K}$-Bessel function, generalizing those of a symmetric cone, which occur when the multiplicity parameter $k$ is chosen to be twice the Peirce dimension constant of the cone. In particular, we show that our $\mathcal{K}$-Bessel function is a solution of a type $B$ system of Dunkl operators, which is a generalization of the Bessel system on a symmetric cone. Furthermore, we define zeta distributions and prove a functional equation, which relates them with their type $B$ Dunkl transform. In the case of symmetric cones, this reduces to the known functional equation between zeta distributions and their Fourier transform. Proving this functional equation, the $\mathcal{K}$-Bessel function and its analytic properties play an essential role. Finally, we examine which of the zeta distributions are given by a positive measure.
\end{abstract}

\maketitle
%%%%%%%%%%%%%%%%%%%%%%%%%%

\section{Introduction}
Zeta integrals are functionals which integrate some function against a complex power $Q^\nu$ of a quadratic form on some underlying space. In particular, they occur in the analysis on symmetric cones, where they depend on a representation of the underlying Jordan algebra, see \cite[Chapter XVI]{FK94} and \cite{C02}. An important special case are matrix cones, where zeta distributions are related to Wishart distributions, see \cite{M82, FK94, R06}. All these zeta distributions have common properties: they can be analytically extended in the parameter $\nu$ and they satisfy a functional equation relating the zeta distributions of a function and the zeta distributions of its Fourier transform. Moreover, zeta distributions on symmetric cones are closely related to Riesz distributions. Important tools in their study are the Hankel transform and the Laplace transform. Since radial analysis on symmetric cones is related to type $A$ Dunkl theory, one expects a generalization to this setting. In fact, the Laplace transform and Riesz distributions in the Dunkl setting have already been studied in \cite{R20, BR22}) and the Hankel transform was considered in \cite{BF98}, but its analytic aspects have not been studied so far. It turns out that the theory of zeta distributions and the Hankel transform on symmetric cones is also closely related to type $B$ Dunkl theory and that the type $A$ Dunkl theory related to Laplace transforms and Riesz distributions will be an essential tool. We will use Laplace transform identities for Jack polynomials, hypergeometric series of Jack polynomials and hypergeometric functions as studied in \cite{BR22}. 
We study different types of Bessel functions which are analogues of Bessel functions on symmetric cones. The usual Bessel function and $\mathcal{K}$-Bessel function was already considered in \cite{H55} for matrix cones, see \cite{FK94} for general symmetric cones. \\
Of central interest in this paper is a two-variable hypergeometric series of Jack polynomials $C_\lambda$ (of $n$ variables and arbitrary index $\alpha$), which we call a Bessel kernel (see Section 3)
$$\mathscr{J}_\nu(w,z)\coloneqq {}_0F_1(\nu;w,-z)=\sum\limits_{\lambda \in \Lambda_+^n} \frac{(-1)^{\abs{\lambda}}}{[\nu]_{\lambda}}\frac{C_\lambda(w)C_\lambda(z)}{\abs{\lambda}! C_\lambda(\underline{1})}, \quad \nu \in \C.$$
It generalizes the kernel of the Hankel transform on a symmetric cone, given explicitly by the Bessel function of the cone, which can be identified as a type $B$ Bessel function (cf. \cite{R07}). We define a non-symmetric counterpart of the Bessel kernel, named $\mathscr{E}_\nu$, which is given by a hypergeometric series of non-symmetric Jack polynomials. The kernel $\mathscr{E}_\nu$ is closely related to the type $B$ Dunkl kernel. The Dunkl-type Hankel transform of a Schwartz function $f_0$ is defined by the kernel $\mathscr{E}_\nu$ (see Section 5)
$$\mathcal{H}_\nu f_0(y)\coloneqq \frac{1}{\Gamma_n(\nu)}\int_{\R^n_+} \mathscr{E}_\nu(-x,y)f_0(x)\Delta(x)^{\nu-\mu_0-1}\omega^A(x) \dif x,$$
where $\Delta(x)=x_1\cdots x_n$ and 
$$\omega^A(x)=\prod\limits_{i<j}\abs{x_i-x_j}^{2k}.$$
It turns out that the transform $\mathcal{H}_\nu$ is given by a Dunkl transform of type $B$ (where the multiplicity depends on the parameter $\nu$)
$$\mathcal{F}^B f(y) = 2^{-n\nu}\mathcal{H}_\nu f_0(\tfrac{y^2}{4}),$$
and $f$ is defined by $f(x)=f_0(x^2)$, here $x^2$ is understood componentwise. That is a known connection for spaces of rectangular matrices, or more generally for symmetric cones, cf. Herz \cite{H55, FK94, R06}. This implies a type $B$ Bernstein identity which was also calculated in \cite{L16}
$$\Delta(T^B)^2\Delta(x^2)^\mu = \mathcal{B}(\mu)\Delta(x^2)^{\mu-1}, \quad \mu \in \C,$$
where $\mathcal{B}$ is a polynomial and $\Delta(T^B)^2$ is the type $B$ Dunkl operator associated to the polynomial $\Delta^2$. Later on in our paper, this Bernstein identity will be important for the analytic continuation of zeta distributions, a family of tempered distributions defined by (see Section 6)
$$\braket{\zeta_{\alpha},f}\coloneqq \frac{1}{\Gamma_n(\alpha)}\int_{\R^n}f(x)\Delta(x^2)^{\alpha-\nu}\omega^B(x) \dif x,$$
for $\alpha \in \C$ with large real part, so that the analytic extension in the parameter $\alpha$ to the whole complex plane $\C$ is obtained in terms of the functional equation 
$$\Delta(T^B)^2\zeta_\alpha=4^nb(\alpha-\nu)\zeta_{\alpha-1},$$
here $b$ is a certain polynomial. This zeta distribution satisfies a characteristic property relating it to its Dunkl transform
$$\zeta_\alpha = 2^{n(2\alpha-\nu)}\mathcal{F}^B\zeta_{\nu-\alpha}.$$
This functional equation is wellknown for spaces of rectangular matrices as well as for symmetric cones, where the Dunkl transform is replaced by a usual Fourier transform, cf. \cite{R06, FK94}. The results about Riesz distributions in \cite{R20} make it possible to explicitly determine those zeta distributions, which are positive measures, so that the corresponding parameter $\alpha \in \C$ lies in a generalized Wallach set. A key tool to verify the functional equation of zeta distributions is the following generalization of the $\mathcal{K}$-Bessel function of a symmetric cone (see Section 4)
$$\mathcal{K}_\nu(w,z)\coloneqq \int_{\R_+^n} E^A(-x,w)E^A(-\tfrac{1}{x},z)\Delta(x)^{\nu-\mu_0-1}\omega^A(x) \dif x,$$
with $E^A$ the type $A$ Dunkl kernel. In fact, after symmetrization over the action of the symmetric group and a particular choice of a type $A$ multiplicity parameter, $\mathcal{K}_\nu$ coincides with a symmetrized version of the $\mathcal{K}$-Bessel function of a symmetric cone, as studied in \cite{C88, D90, FK94}. The Bessel function $\mathcal{K}_\nu$ is defined for all $\nu \in \C$, $w,z \in \C^n$ with $\Re\, w, \, \Re\, z>0$ and satisfies
$$\mathcal{K}_\nu(w,z)=\mathcal{K}_{-\nu}(z,w) \te{ and } \abs{\mathcal{K}_\nu(w,z)}\le \mathcal{K}_{\Re\, \nu}(\Re\, w, \Re\, z).$$
We shall also introduce a generalized $\mathcal{K}$-Bessel function with a multivariate index $\nu \in \C^n$.
The Dunkl operator $\Delta(T^A)$ acts on the $\mathcal{K}$-Bessel function by shifting the parameter $\nu$ and we prove growth condition on $\mathcal{K}_\nu$. \\
Finally, on symmetric cones, the $\mathcal{K}$-Bessel function defines an eigenfunction of the system of Bessel operators. Similarly, we obtain in the presented paper that the Dunkl-type $\mathcal{K}$-Bessel function is (up to squared variables) an eigenfunction of $\Z_2^n$-invariant type $B$ Dunkl operators, which play the role of the Bessel operators on symmetric cones.

\section{The rational Dunkl setting}\label{Dunkl}
For a general background on rational Dunkl theory, the reader is referred to \cite{Dun89, dJ93, R03, DX14}. 
We equip the $n$-dimensional Euclidean space $\R^n$ with the usual inner product $\braket{x,y}=\sum_{i=1}^n x_iy_i$ and extend this naturally to a $\C$-bilinear form on $\C^n$. In $\R^n$ we consider the root systems $R \in \set{A,B}$ with 
\begin{align*}
A&\coloneqq A_{n-1}=\set{\pm(e_i-e_j) \mid 1\le i< j \le n}, \\
B&\coloneqq B_n = \set{\pm e_i \mid 1\le i \le n} \cup \set{\pm(e_i \pm e_j) \mid 1\le i < j \le n},
\end{align*}
where $(e_i)_{1\le i \le n}$ is the canonical basis of $\R^n$. The corresponding Weyl groups are
$$W_{\! A}\coloneqq \mathcal{S}_n, \quad W_{\! B} \coloneqq \mathcal{S}_n \ltimes \Z_2^n,$$
where the symmetric group $\mathcal{S}_n$ acts by permutation of coordinates and $\Z_2^n$ acts by sign changes of coordinates. A multiplicity function on $R$ is a $W_{\! R}$-invariant function $\kappa_R:R \to \C$. Moreover, $W_R$ acts on functions $f:\R^n \to \C$ by the assignment $w.f=x\mapsto f(w^{-1}x)$. We consider multiplicity functions of the specific form
$$\kappa_A=k, \quad \kappa_B=(k,k'),$$
where $k$ is the value on $\pm(e_i\pm e_j)$ and $k'$ is the value on $e_i$. The (rational) Dunkl operator associated to $(R,\kappa_R)$ into direction $\xi \in \R^n$ is
$$T_\xi^R=T_\xi^R(\kappa_R) \coloneqq \partial_\xi + \frac{1}{2}\sum\limits_{\alpha \in R} \kappa_R(\alpha)\braket{\alpha,\xi}\frac{1-s_\alpha}{\braket{\alpha,\m}},$$
where $s_\alpha(x)=x-2\tfrac{\braket{\alpha,x}}{\braket{\alpha,\alpha}}\alpha$ is the reflection in the hyperplane perpendicular to $\alpha$. The Dunkl operators commute for fixed $(R,\kappa_R)$, i.e. $T_\xi^RT_\eta^R=T_\eta^RT_\xi^R$ on $C^2(\Omega)$ for all $\xi,\eta \in \R^n$ and $W_{\! R}$-invariant open sets $\Omega \tm \R^n$, so that $p(T^R)$ is defined for any polynomial $p$ on $\R^n$. For abbreviation, we write $T_i^R\coloneqq T_{e_i}^R$. If $\Re \, \kappa_R \ge 0$, the Dunkl operators define an $W_{\! R}$-invariant symmetric non-degenerated bilinear form on the space of polynomials $\mathcal{P}\coloneqq\C[\R^n]$ by $[p,q]^R\coloneqq p(T^R)q(0)$, which is positive definite on $\R[\R^n]$ for $\kappa_R\ge 0$. Moreover, polynomials of different homogeneous degree are orthogonal with respect to this pairing. \\
If $\Re \, \kappa_R \ge 0$, then for each $\lambda \in \C^n$, the joint eigenvalue problem
$$\begin{cases}
T_\xi^R f = \braket{\lambda,\xi}f, & \te{ for all }\xi \in \R^n \\
f(0)=1
\end{cases}$$
has an unique holomorphic solution $E^R(\lambda,\m)=E_{\kappa_R}^R(\lambda,\m):\C^n \to \C$, called the Dunkl kernel associated to $(R,\kappa_R)$. The Dunkl kernel is positive on $\R^n\times \R^n$ (if in addition $\kappa_R\ge 0$) and satisfies
$$E^R(w\lambda,wz)=E^R(\lambda,z), \quad E^R(s\lambda,z)=E^R(\lambda,sz), \quad E^R(\lambda,z)=E^R(z,\lambda),$$ 
for all $\lambda,z \in \C^n, \, w \in W_{\! R}$ and $s \in \C$. Moreover, the Bessel function associated to $(R,\kappa_R)$ is defined as
$$J^R(\lambda,z)=J^R_{\kappa_R}(\lambda,z)\coloneqq \frac{1}{\# W_{\! R}}\sum\limits_{w \in W_{\! R}} E^R(\lambda,wz).$$
An important weight function in Dunkl theory is 
$$\omega^R(x)=\omega_{\kappa_R}^R(x)=\prod\limits_{\alpha \in R} \abs{\braket{x,\alpha}}^{\kappa_R(\alpha)}.$$
The weight functions associated to the root systems $A$ and $B$ are related to each other by the equation
\begin{equation}\label{WeightRelation}
\omega^B(x)=\Delta(x^2)^{k'}\omega^A(x^2),
\end{equation}
where $x^2$ is understood componentwise and
\begin{equation}\label{Det}
\Delta(x)\coloneqq x_1\cdots x_n.
\end{equation} 
This weight occurs in the Dunkl transform, defined for $f \in L^1(\R^n,|\omega^R(x)|\d x)$ by
$$\mathcal{F}^Rf (\xi)\coloneqq \mathcal{F}_{\kappa_R}^Rf(\xi) \coloneqq \frac{1}{c_R} \int_{\R^n} E^R(-i\xi,x)f(x)\omega^R(x) \dif x,$$
with the constant
$$c_R=\int_{\R^n} e^{-\abs{x}^2/2}\omega^R(x) \dif x,$$
 which defines an automorphism of the Schwartz space $\mathscr{S}(\R^n)$. In particular, for $\kappa_R=0$, the Dunkl transform coincides with the usual Euclidean Fourier transform. For $f \in \mathscr{S}(\R^n)$ we have
\begin{equation}\label{DunklTrafo}
T_\xi^R \mathcal{F}^R = -\mathcal{F}^R m_{i\xi} \quad  \te{ and } \quad \mathcal{F}^RT_\xi^R =m_{i\xi}\mathcal{F}^R,
\end{equation}
with the multiplication operator $m_z \coloneqq f \mapsto \braket{z,\m}f$ for $z \in \C^n$. \\
To improve readability, we will omit the subscript $\kappa_R$ throughout the paper.

\section{The type A Dunkl-Laplace transform}\label{Laplace}
First, we summarize some facts about the Dunkl-Laplace transform from \cite{R20, BR22}. Throughout this section, we consider a non-negative multiplicity $\kappa_A = k \ge 0$. \\
The (type A) Dunkl-Laplace transform of a locally integrable function $f \in L^1_{\te{loc}}(\R^n_+)$ is defined on the cone $\R_+^n=]0,\infty[^n$ by
$$\mathcal{L}f(z)\coloneqq \int_{\R_+^n} E^A(-x,z)f(x) \omega^A(x) \dif x$$
for $z \in \C^n$, provided the integral exists. 
Most of the analytic aspects of this transform are based on the following properties of the type $A$ Dunkl kernel:
\begin{equation}\label{DunklKernelEstimate}
\begin{aligned}
E^A(\underline{s}+w,z)&=e^{\braket{\underline{s},w}}E^A(w,z), &w,z \in \C^n,\, s \in \C, \\
\abs{E^A(-x,z)}&\le e^{-\min_i \Re\, z_i\m \nrm{x}_1}, &x \in \R^n_+, \Re\, z>0, \\
\abs{E^A(-x,z)}&\le E^A(-x,\Re \, z), &x \in \R^n, \Re\, z>0, \\
0&<E^A(-x,y)\le 1, &x,y \in \R_+^n,
\end{aligned}
\end{equation}
where $\nrm{x}_1\coloneqq\sum_i\abs{x_i}$ and $\underline{s}\coloneqq(s,\ldots,s)$. For us, the following properties of $\mathcal{K}$ will be relevant:

\begin{theorem}[\cite{R20}]\label{LaplaceProperties}
Consider $f \in L^1_{\te{loc}}(\R_+^n)$. Then:
\begin{enumerate}[leftmargin=0.8cm, itemsep=5pt]
\item[\rm{(i)}] If $\mathcal{L}f(a)$ exists for $a \in \R^n$, then $\mathcal{L}f$ exists and defines a holmorphic function on $H_n(a)\coloneqq \set{z \in \C^n \mid \Re\, z>a}$, where $\Re\, z>a$ is defined componentwise. Moreover, for any polynomial $p\in \C[\R^n]$,
$$p(-T^A)\mathcal{L}f(z)=\mathcal{L}(pf)(z), \quad z \in H_n(a).$$
\item[\rm{(ii)}] If $\abs{f(x)}\le e^{-s\nrm{x}_1}$ for some $s \in \R$, then $\mathcal{L}f$ exists on $H_n(\underline{s})$.
\item[\rm{(iii)}] If $\mathcal{L}f(\underline{s})$ exists for some $s \in \R$ and
$y \mapsto \mathcal{L}f(\underline{s}+iy) \in L^1(\R^n,\omega^A(x)\d x),$
then $f$ has a continuous representative $f_0$ such that
$$\frac{(-i)^n}{c_A^2}\int_{\Re\, z = \underline{s}} \mathcal{L}f(z)E^A(x,z)\omega^A(z) \dif z = \begin{cases}
f_0(x) & \te{if } x>0 \\
0 &  \te{otherwise.}
\end{cases}$$
The integral is understood as an $n$-fold line integral.
\end{enumerate}
\end{theorem}

We consider non-symmetric Jack polynomials $(E_\eta)_{\eta \in \N_0^n}$ and symmetric Jack polynomials $(P_\lambda)_{\lambda \in \Lambda_+^n}$ (of index $\alpha=\tfrac{1}{k}$) with $\Lambda_+^n=\set{\lambda \in \N_0^n \mid \lambda_1\ge \ldots \lambda_n \ge 0}$, the partitions of length at most $n$. We assume monic normalization, i.e. $E_\eta(x)=x^\eta+...$ and $P_\lambda(x)=x^\lambda+...$, for $\eta \in \N_0^n$ and $\lambda \in \Lambda_+^n$. For the precise definition of these polynomials the reader is referred to \cite{S89,S98, BF98, SZ07, F10}. 

\begin{theorem}[\cite{BR22}]\label{JackLaplace}
Consider $\mu \in \C$ with $\Re \, \mu > \mu_0\coloneqq k(n-1)$. The Jack polynomials satisfy the following Laplace transform identities for $\Re\, z>0$:
\begin{align*}
\int_{\R_+^n} E^A(-x,z)E_\eta(x)\Delta(x)^{\mu-\mu_0-1} \omega^A(x) \dif x &= \Gamma_n(\underline{\mu}+\eta_+)E_\eta(\tfrac{1}{z})\Delta(z)^{-\mu}, \\
\int_{\R_+^n} E^A(-x,z)P_\lambda(x)\Delta(x)^{\mu-\mu_0-1} \omega^A(x) \dif x &= \Gamma_n(\underline{\mu}+\lambda)P_\lambda(\tfrac{1}{z})\Delta(z)^{-\mu},
\end{align*}
where $\eta_+$ is the unique partition contained in the orbit $\mathcal{S}_n\eta$, and $\Gamma_n$ is the generalized gamma function
$$\Gamma_n(z)\coloneqq c_A(2\pi)^{-\tfrac{n}{2}}\prod\limits_{j=1}^n \Gamma(z_j-k(j-1)).$$
In particular,
$$\int_{\R_+^n}E^A(-x,z)\Delta(x)^{\mu-\mu_0-1}\omega^A(x)\dif x = \Gamma_n(\underline{\mu})\Delta(z)^{-\mu}.$$
\end{theorem}

There is a well-known renormalization of the symmetric Jack polynomials, called the $C$-normalization and denoted by $(C_\lambda)_{\lambda \in \Lambda_+^n}$. As well as a similar renormalization $(L_\eta)_{\eta \in \N_0^n}$ of the non-symmetric Jack polynomials, such that
$$\sum\limits_{\substack{\eta \in \N_0^n \\ \abs{\eta}=p}}L_\eta(x)=\sum\limits_{\substack{\lambda \in \Lambda_+^n \\ \abs{\lambda}=p}}C_\lambda(x)=(x_1+\ldots+x_n)^p;$$
see for instance \cite{S89, F10, BR22}.
For parameters $\mu \in \C^p$ and $\nu \in \C^q$ with $\nu_i \notin \set{0,k,\ldots, k(n-1)}-\N_0$, for all $1\le i\le q$, we define the Jack-hypergeometric series (cf. \cite{BF98, M13, BR22}):
\begin{align*}
{}_pK_q(\mu;\nu;w,z) &\coloneqq \sum\limits_{\eta \in \N_0^n} \frac{[\mu_1]_{\eta_+}\cdots [\mu_p]_{\eta_+}}{[\nu_1]_{\eta_+}\cdots [\nu_q]_{\eta_+}} \frac{L_\eta(w)L_\eta(z)}{\abs{\eta}!\, L_\eta(\underline{1})}, \\
{}_pF_q(\mu;\nu;w,z) &\coloneqq \sum\limits_{\lambda \in \Lambda_+^n} \frac{[\mu_1]_{\lambda}\cdots [\mu_p]_{\lambda}}{[\nu_1]_{\lambda}\cdots [\nu_q]_{\lambda}} \frac{C_\lambda(w)C_\lambda(z)}{\abs{\lambda}! \, C_\lambda(\underline{1})},
\end{align*}
with the generalized Pochhammer symbol 
$$[\alpha]_\lambda\coloneqq \frac{\Gamma_n(\underline{\alpha}+\lambda)}{\Gamma_n(\underline{\alpha})}, \te{ for } \alpha \in \C, \,\lambda \in \Lambda_+^n.$$ 
For abbreviation we shall write
$$\Gamma_n(\alpha)=\Gamma_n(\underline{\alpha})\te{ for } \alpha \in \C \;\te{ and } \;[\mu]_\lambda=[\mu_1]_\lambda \cdots [\mu_p]_\lambda \te{ for } \mu \in \C^p.$$
Considering $\mathcal{S}_n$-means, the two types of hypergeometric series are related by
$$\frac{1}{n!}\sum\limits_{\sigma \in \mathcal{S}_n} {}_pK_q(\mu;\nu;w,\sigma z)={}_pF_q(w,z).$$

\begin{theorem}[\cite{BR22}]\label{HypergeometricSeries}
For $\mu \in \C^p$, $\nu \in \C^q$ with $\nu_i \notin \set{0,k,\ldots, k(n-1)}-\N_0$ for $1\le i\le q$, the following hold.
\begin{enumerate}[leftmargin=0.8cm, itemsep=5pt]
\item[\rm{(i)}] If $p\le q$, then ${}_pK_q(\mu;\nu;\m,\m)$ and ${}_pF_q(\mu;\nu;\m,\m)$ are entire. Moreover, they are holomorphic in $(\mu,\nu)$.
\item[\rm{(ii)}] If $p=q+1$, then ${}_pK_q(\mu;\nu;\m,\m)$ and ${}_pF_q(\mu;\nu;\m,\m)$ are holomorphic on the domain $\set{(w,z) \in \C^n\times \C^n \mid \nrm{w}_\infty \m \nrm{z}_\infty<1}$. Again, they are holomorphic in $(\mu,\nu)$.
\item[\rm{(iii)}] For $p\le q$ and $\Re\, \mu' > \mu_0$,
\begin{align*}
\quad \quad \int_{\R_+^n}E^A(-x,z)\, {}_pK_q(\mu;\nu;w,x) &\Delta(x)^{\mu'-\mu_0-1} \omega^A(x) dx \\
&=\Gamma_n(\mu')\Delta(z)^{-\mu'}\, {}_{p+1}K_q((\mu',\mu);\nu;w,\tfrac{1}{z}).
\end{align*}
for all $w,z \in \C^n$ with $\Re \, z>0$ if $p<q$ and additionally $\nrm{w}_\infty\m \nrm{\tfrac{1}{z}}_\infty<\tfrac{1}{n}$ if $p=q$. The same is true for ${}_pF_q$ instead of ${}_pK_q$.
\item[\rm{(iv)}] For $p\le q+1$ and $w \in \C^n$
$$\quad \quad \Delta(T){}_pK_q(\mu;\nu;w,\m) = \frac{[\mu]_{\underline{1}}}{[\nu]_{\underline{1}}}\Delta(w)\, {}_{p}K_q(\mu+\underline{1};\nu+\underline{1};w,\m).$$
The same is true for ${}_pF_q$ instead of ${}_pK_q$.
\item[\rm{(v)}] ${}_0K_0(w,z)=E^A(w,z)$ and ${}_0F_0(w,z)=J^A(w,z)$, as already observed in \cite{BF98}.
\end{enumerate}
\end{theorem}

\section{Bessel kernel and $\mathcal{K}$-Bessel function}\label{Bessels}
In this section, we define a $\mathcal{K}$-Bessel function and a Bessel kernel for the root system $A_{n-1}$. All this will be done in line with radial analysis on symmetric cones. \\
To become more precise, consider an irreducible symmetric cone $\Omega=G/K$ with associated Euclidean Jordan algebra $V$ of dimension $m$, rank $n$ and Peirce dimension constant $d$. For $k=\tfrac{d}{2}$, the Bessel function of index $\nu \in \C\setminus (\set{k,\ldots,k(n-1)}-\N_0)$ associated to $\Omega$ is defined as the following entire function on $V_\C$, the complexification of $V$
$$\mathcal{J}_\nu^\Omega(z)={}_0F_1^\Omega(\nu,-z),$$
where ${}_0F_1^\Omega$ is a hypergeometric series of spherical polynomials (cf. \cite{FK94}). Since $\mathcal{J}_\nu^\Omega$ is $K$-invariant, it can be considered as function of the spectrum of $z$, so that $\mathcal{J}_\nu^\Omega$ becomes a two-variable ${}_0F_1$ series of Jack polynomials as introduced in Section \ref{Laplace}, with one variable set to $\underline{1}$. This connection is verified in the subsequent remark. It is an important observation, that this Bessel function, more precisely the two-variable ${}_0F_1$ series, is essentially a type $B$ Bessel function. This is one reasons why the type $B$ Dunkl theory shows up in this setting. 

\begin{remark}\label{HypGeoSeries}
Let $(Z_\lambda)_{\lambda \in \Lambda_+^n}$ be the spherical polynomials of $\Omega$, normalized such that
$$(\tr \, x)^p = \sum\limits_{\substack{\lambda \in \Lambda_+^n \\ \abs{\lambda}=p}}Z_\lambda(x),$$
where $\tr$ is the Jordan trace. Then $Z_\lambda(x)=C_\lambda(\mathrm{spec}\, x)$, where $\mathrm{spec} \, x \in \R^n$ denotes the spectrum of $x$ in decreasing order. A hypergeometric series associated to $\Omega$ is of the form
$${}_pF_q^\Omega(\mu;\nu;x)=\sum\limits_{\lambda \in \Lambda_+^n} \frac{[\mu]_\lambda}{[\nu]_\lambda \abs{\lambda}!} Z_\lambda(x)={}_pF_q(\mu;\nu,\mathrm{spec} \, x).$$
There is no canonical terminology of a two-variable hypergeometric series of a symmetric cone, but consider for a moment the function
$${}_pF_q^\Omega(\mu;\nu;x,y)\coloneqq {}_pF_q^\Omega(\mu;\nu;P(\sqrt{x})y),$$
where $P$ is the quadratic representation of the Jordan algebra $V$. We note that ${}_0F_1^\Omega$ occurs as kernel of the Hankel transform of the symmetric cone, see \cite{FK94}. By \cite[Corollary XI.3.2]{FK94}, the $K$-mean of ${}_pF_q^\Omega$ is computed as
\begin{align*}
\int_K {}_pF_q^\Omega(\mu;\nu;x,ky) \dif k &=\sum\limits_{\lambda \in \Lambda_+^n} \frac{[\mu]_\lambda}{[\nu]_\lambda \abs{\lambda}!} \int_K Z_\lambda(P(\sqrt{x})ky) \dif k \\
&=\sum\limits_{\lambda \in \Lambda_+^n} \frac{[\mu]_\lambda}{[\nu]_\lambda}\frac{Z_\lambda(x)Z_\lambda(y)}{\abs{\lambda}!Z_\lambda(e)}={}_pF_q(\mu;\nu;\mathrm{spec}\, x , \mathrm{spec}\, y).
\end{align*}
for arbitrary $x \in \Omega$.
\end{remark}

\begin{definition}\label{Besselkernels}
For $\nu \in \C$ with $\nu \notin \set{0,k,\ldots,k(n-1)}-\N_0$, we define the type $A$ (non-)symmetric Bessel kernels as
\begin{align*}
\mathscr{E}_\nu(w,z)&\coloneqq {}_0K_1(\nu;w,-z)=\sum\limits_{\eta \in \N_0^n} \frac{(-1)^{\abs{\eta}}}{[\nu]_{\eta_+}} \frac{L_\eta(w)L_\eta(z)}{\abs{\eta}! \, L_\eta(\underline{1})}, \\
\mathscr{J}_\nu(w,z)&\coloneqq {}_0F_1(\nu;w,-z)=\sum\limits_{\lambda \in \Lambda_+^n} \frac{(-1)^{\abs{\lambda}}}{[\nu]_{\lambda}} \frac{C_\lambda(w)C_\lambda(z)}{\abs{\lambda}!\,  C_\lambda(\underline{1})},
\end{align*}
which are entire functions in the variables $w,z$ by Theorem \ref{HypergeometricSeries} and holomorphic in the parameter $\nu\in \C$ satisfying the condition above. In particular, the Bessel kernels are related via
$$\frac{1}{n!}\sum\limits_{\sigma \in \mathcal{S}_n} \mathscr{E}_\nu(w,\sigma z)=\mathscr{J}_\nu(w,z).$$
\end{definition}

From now on, we consider a fixed parameter $\nu \in \C$. The type $A$ multiplicity is
$$\kappa_A=k \ge 0,$$ 
and the type $B$ multiplicity is
$$\kappa_B=(k,k')\quad \te{with} \quad k'=\nu-\mu_0-\tfrac{1}{2}, \quad \te{where} \quad \mu_0\coloneqq k(n-1).$$ Hence, we have a one-to-one correspondence between the pairs $(\kappa_A,\nu)$ and the type $B$ multiplicities $\kappa_B$. The following proposition justifies the term "Bessel-kernel" for the functions $\mathscr{J}_\nu$ and $\mathscr{E}_\nu$.

\begin{proposition}\label{BesselKernelTypeBConnection}
The type $B$ Dunkl kernel and Bessel function satisfy
\begin{align*}
{}_0K_1(\nu;\tfrac{w^2}{2},\tfrac{z^2}{2})&=\frac{1}{2^n}\sum\limits_{\tau \in \Z_2^n} E^B(w,\tau z),\\
{}_0F_1(\nu;\tfrac{w^2}{2},\tfrac{z^2}{2})&=J^B(w,z),
\end{align*}
In particular,
\begin{align*}
\mathscr{E}_\nu(\tfrac{w^2}{2},\tfrac{z^2}{2})&=\frac{1}{2^n}\sum\limits_{\tau \in \Z_2^n} E^B(iw,\tau z), \\
\mathscr{J}_\nu(\tfrac{w^2}{2},\tfrac{z^2}{2})&=J^B(iw,z).
\end{align*}
\end{proposition}

We remark, that the statement on the Bessel kernel was already observed in \cite[Proposition 4.5]{R07}.

\begin{proof}
The proof is similar to \cite{R07} for the Bessel function. Using \cite[Proposition 4.18]{BF98}, we obtain the following formula for the Dunkl pairing $[\m,\m]^B$ of non-symmetric Jack polynomials $E_\eta$
$$E_\eta((T_x^B)^2)E_\mu(x^2)\Big|_{x=0}=[E_\eta(x^2),E_\mu(x^2)]^B=\begin{cases}
4^{\abs{\eta}}[\nu]_{\eta_+}k^{\abs{\eta}} \tfrac{d_\eta'e_\eta}{d_\eta}, & \te{ if }\eta=\mu \\
0, & \te{ otherwise}
\end{cases}$$
with certain constants $d_\eta,d_\eta', e_\eta$, satisfying $\tfrac{e_\eta}{d_\eta}=E_\eta(\underline{1})$, cf. \cite[Formula (12.3.3)]{F10}. Therefore, by definition of the renormalization $L_\eta$ of the non-symmetric Jack polynomials (cf. \cite{BR22}),
$$[L_\eta(x^2),L_\mu(x^2)]^B =  4^{\abs{\eta}}\abs{\eta}!\,[\nu]_{\eta_+}L_\eta(\underline{1})\m \delta_{\eta\mu},$$
where $\delta_{\mu\eta}$ is the Kronecker delta. Since $(L_\eta)_{\eta \in \N_0^n}$ is a homogeneous basis for $\C[\R^n]$, we have an expression
$$\frac{1}{2^n}\sum\limits_{\tau \in \Z_2^n} E^B(w,\tau z)=\sum\limits_{\mu \in \N_0^n} a_\mu(w)L_\mu(z^2),$$
with certain coefficients $a_\mu(w) \in \C$. Together with the $\Z_2^n$-invariance of $L_\mu(x^2)$, the eigenvalue equation for the Dunkl kernel leads to
\begin{align*}
L_\eta(w^2) &= \frac{1}{2^n}\sum\limits_{\tau \in \Z_2^n} L_\eta((\tau w)^2) E^B(\tau w,z)\Big|_{z=0} = \frac{1}{2^n}\sum\limits_{\tau \in \Z_2^n} L_\eta((T_z^B)^2) E^B(\tau w,z)\Big|_{z=0} \\
&=\sum\limits_{\mu \in \N_0^n} a_\mu(w)[L_\eta(z^2),L_\mu(z^2)]^B = a_\eta(w)4^{\abs{\eta}}\abs{\eta}!\,[\nu]_{\eta_+}L_\eta(\underline{1}).
\end{align*}
Finally, equating coefficients leads to the stated formula.
\end{proof}

\begin{lemma}\label{BesselKernelLaplace}
Let $w,z \in \C^n$ with $\Re \, z>0$ and $\Re \, \nu >\mu_0$. Then
\begin{align*}
\int_{\R_+^n} E^A(-x,z)\mathscr{E}_\nu(w,x)\Delta(x)^{\nu-\mu_0-1} \omega^A(x) \dif x&=\Gamma_n(\nu)E^A(w,-\tfrac{1}{z})\Delta(z)^{-\nu}, \\
\int_{\R_+^n} J^A(-x,z)\mathscr{J}_\nu(w,x)\Delta(x)^{\nu-\mu_0-1} \omega^A(x) \dif x&=\Gamma_n(\nu)J^A(w,-\tfrac{1}{z})\Delta(z)^{-\nu}.
\end{align*}
\end{lemma}

\begin{proof}
This is immediate from part (iii) of Theorem \ref{HypergeometricSeries} and the observation
$${}_1K_1(\nu;\nu;w,z)={}_0K_0(w,z)=E^A(w,z).$$
\end{proof}

\begin{theorem}[Integral representation]\label{BesselKernelIntegral}
Let $\Re \, \nu >2\mu_0+1, w \in \C^n$ and $x \in \R_+^n$. Then for all $s \in \R_+$, we have
\begin{align*}
\mathscr{E}_\nu(w,x)\Delta(x)^{\nu-\mu_0-1} &= \frac{\Gamma_n(\nu)}{c_A^2\, i^n} \int_{\Re(\zeta)=\underline{s}} E^A(x,\zeta)E^A(w,-\tfrac{1}{\zeta})\Delta(\zeta)^{-\nu}\omega^A(\zeta) \, \d \zeta, \\
\mathscr{J}_\nu(w,x)\Delta(x)^{\nu-\mu_0-1} &= \frac{\Gamma_n(\nu)}{c_A^2\, i^n} \int_{\Re(\zeta)=\underline{s}} J^A(x,\zeta)J^A(w,-\tfrac{1}{\zeta})\Delta(\zeta)^{-\nu}\omega^A(\zeta) \, \d \zeta.
\end{align*}
\end{theorem}

\begin{proof}
Due to the Cauchy-type inversion formula of the Dunkl-Laplace transform (Theorem \ref{LaplaceProperties} (iii)), it suffices to prove that the right hand side of Lemma \ref{BesselKernelLaplace} is integrable as a function of $z$ over $\underline{s}+i\R^n$ for $s>0$ with respect to the measure $\omega^A(x)\d x$. The map $z \mapsto \tfrac{1}{z}$ is bounded on $\underline{s}+i\R^n$ and so is $z\mapsto E^A(w,-\tfrac{1}{z})$. Thus, we only have to verify the integrability condition for $\Delta(z)^{-\nu}$. Since $\omega^A(x)=\abs{D(x)}^{2k}$ with $D(x)=\prod_{i<j}(x_i-x_j)$ of degree $(n-1)$ in $x_i$, the function
$$y \mapsto \Delta(\underline{s}+y^2)^{-\tfrac{1}{2}\Re \, \nu}\omega^A(y)$$
is integrable over $\R^n$ iff $\Re \, \nu>2\mu_0+1$, since $\sqrt{\underline{s}+y_i^2} \sim \abs{y_i}$ for large $\abs{y_i}$. Hence, the claim holds for $\Re \, \nu > 2\mu_0+1$.
\end{proof}

\begin{corollary}[Recurrence formulas]\label{RecurrenceBessel}
For $w,z \in \C^n$ and $\nu \in \C$ with $\Re \, \nu >\mu_0$,
\begin{enumerate}[leftmargin=0.8cm, itemsep=5pt]
\item[\rm{(i)}] $\displaystyle \Delta(T_z^A)\mathscr{E}_\nu(w,z)=\frac{(-1)^n}{[\nu]_{\underline{1}}}\Delta(w)\mathscr{E}_{\nu+1}(w,z)$.
\item[\rm{(ii)}]  In addition $\Re\, z > 0$, then $$\Delta(T_z^A)\left(\mathscr{E}_\nu(w,z)\Delta(z)^{\nu-\mu_0-1}\right)=[\nu-1]_{\underline{1}}\, \mathscr{E}_{\nu-1}(w,z)\Delta(z)^{\nu-\mu_0-2}.$$
\end{enumerate}
The same formulas is true if $\mathscr{E}_\nu$ is replaced by $\mathscr{J}_\nu$. 
In particular, due to the second recurrence property, $\mathscr{E}_\nu$ and $\mathscr{J}_\nu$ can be analytically extended to $\nu \in \C$ as functions on $\C^n\times \set{\Re \, z>0}$, so that (i) and (ii) are both valid for arbitrary $\nu \in \C$.
\end{corollary}

\begin{proof}
\
\begin{enumerate}[leftmargin=0.8cm, itemsep=5pt]
\item Denoting $f^-(x)=f(-x)$, we have $T_\xi^Af^- = - (T_\xi^A f)^-$. Hence, part (i) is a consequence of Theorem \ref{HypergeometricSeries} (iv).
\item By analyticity in $\nu$ and $z$, we can assume that $\Re \, \nu >2\mu_0+2$ and $z=x \in \R_+^n$. The claim can be deduced from the integral representation of the Bessel kernels in Theorem \ref{BesselKernelIntegral}. By changing the order of differentation and integration, the stated formula follows from $\Delta(T_x^A)E^A(x,\zeta)=\Delta(\zeta)E^A(x,\zeta)$ and $[\nu-1]_{\underline{1}}=\tfrac{\Gamma_n(\nu)}{\Gamma_n(\nu-1)}$. By \cite[Proposition 2.6]{R03} we have
$$\abs{\partial_\xi^x E^A(x,\underline{s}+iy)} \le (s+\abs{y})e^{\braket{x,\underline{s}}},$$
which is locally bounded in $x$. By similar computations as in Theorem \ref{BesselKernelIntegral}, the condition $\Re \, \nu >2\mu_0+2$ shows that the function 
$$y\mapsto \partial_\xi^x E^A(x,\underline{s}+iy)\Delta(\underline{s}+iy)^{\nu-\mu_0-1}\omega^A(y)$$
is dominated on $\R^n$ by
$$y \mapsto e^{\braket{x,\underline{s}}}(s+\abs{y})\Delta(\underline{s}+iy)^{\nu-\mu_0-1}\omega^A(y), $$
which is locally bounded in $x$ and integrable over $\R^n$. This justifies the change of order of differentation and integration.
\end{enumerate}
\end{proof}

We now come to the definition of the Dunkl-type $\mathcal{K}$-Bessel function which shares important properties with the $\mathcal{K}$-Bessel functions on symmetric cones. Recall, that on a symmetric cone $\Omega$ the $\mathcal{K}$-Bessel function is defined for $\mathbf{s}\in \C^n$ as
\begin{equation}\label{ConeKBessel}
\mathcal{K}_{\mathbf{s}}^\Omega(x,y)=\int_\Omega e^{-\braket{x,u^{-1}}-\braket{y,u}}\Delta_{\mathbf{s}}(u)\det(u)^{-\tfrac{m}{n}} \dif u,
\end{equation}
where $\Delta_{\mathbf{s}}$ is the generalized power function of $\Omega$ (which coincides with $\det^s$ if $\mathbf{s}=(s,\ldots,s)$), see \cite[Chapter XVI, Section 3]{FK94}. The $\mathcal{K}$-Bessel function for arbitrary symmetric cones was first considered in \cite{C88} and further studied in \cite{D90}.

\begin{definition}\label{DefKBessel}
The Dunkl-type $\mathcal{K}$-Bessel of index $\nu \in \C$ is defined by
$$\mathcal{K}_\nu(w,z)\coloneqq \int_{\R_+^n} E^A(-x,w)E^A(-\tfrac{1}{x},z)\Delta(x)^{\nu-\mu_0-1} \omega^A(x) \dif x.$$
The convergence and further properties of the integral will be investigated in the subsequent theorem.
\end{definition}

By a change of variables, the following Lemma is immediate.

\begin{lemma}\label{InvariantMeasure}
On $\R_+^n$, the $\mathcal{S}_n$-invariant measure 
\begin{equation}\label{InvariantMeasureEq}
\Delta(x)^{-\mu_0-1}\omega^A(x) \d x
\end{equation}
is invariant under the transformation $x\mapsto \tfrac{1}{x}$ as well as under $x \mapsto sx$ for $s>0$. 
\end{lemma}
This invariance will be highly relevant in the subsequent results. The measure \eqref{InvariantMeasureEq} has to be understood as the Dunkl analogue of the invariant measure of a symmetric cone. In fact, if $k$ is related to a symmetric cone $\Omega=G/K$, then the measure \eqref{InvariantMeasureEq} is the radial part of the $G$-invariant measure on $\Omega$, cf. \cite[Theorem VI.2.3]{FK94}.

\begin{remark}\label{KBesselCone}
Consider the situation of an irreducible symmetric cone $\Omega=G/K$ inside a simple Euclidean Jordan algebra $V$ of dimension $m$, rank $n$ and with Peirce dimension constant $d$. Fix a Jordan frame $(c_1,\ldots,c_n)$, write $k=\tfrac{d}{2}$ and denote by 
$$\mathrm{spec}\, x\in C_+\coloneqq \set{\xi \in \R_+^n \mid \xi_1 \ge \ldots \ge \xi_n}$$ the spectrum of $x \in \Omega$ or, more generally $\mathrm{spec}\, x \in \R^n$ for $x \in V$.
By \cite[Remark 3.2]{R20}, the $K$-mean of the $\mathcal{K}$-Bessel function \eqref{ConeKBessel} with the particular index $\mathbf{s}=\underline{\nu}=(\nu,\ldots,\nu)$ becomes
\begin{align*}
(\mathcal{K}_{\underline{\nu}}^\Omega)^{\#}(x,y) &\coloneqq \int_\Omega \left( \int_K e^{-(kx,u^{-1})} \dif k\right) e^{-(y,u)} \det(u)^{-\nu-\tfrac{m}{n}} \dif u \\
&=\int_\Omega J^A(-\mathrm{spec}\, x, \spec \,u^{-1}) e^{-(y,u)} \det(u)^{-\nu-\frac{m}{n}} \dif u \\
&=c_0\int_{C_{++}}J^A(-\mathrm{spec}\, x, \tfrac{1}{\xi})\left( \int_Ke^{-(y,k\xi)}\; dk\right) \Delta(\xi)^{-\nu-\mu_0-1}\omega^A(\xi)\; d\xi \\
&= c_0 \int_{C_{++}}J^A(-\mathrm{spec}\, x, \tfrac{1}{\xi})J^A(-\mathrm{spec}\, y, \xi)\Delta(\xi)^{-\nu-\mu_0-1}\omega^A(\xi)\; d\xi \\
&= \frac{c_0}{n!}\sum_{\sigma \in \mathcal{S}_n} \mathcal{K}_{-\nu}(\mathrm{spec}\, y, \sigma\mathrm{spec}\, x)=\frac{c_0}{n!}\sum_{\sigma \in \mathcal{S}_n} \mathcal{K}_\nu(\sigma\mathrm{spec}\, x,\mathrm{spec} y),
\end{align*}
where $C_{++}=\set{\xi \in \R^n \mid \xi_1 \ge \ldots \xi_n\ge 0}$ and the last equality is verified in the next theorem.
Hence, this definition of the $\mathcal{K}$-Bessel function in the Dunkl setting is consistent with that on symmetric cones up to a constant factor and symmetrization. In particular, note that the one variable versions are related by
$$\mathcal{K}_{\underline{\nu}}^\Omega(x,e)=\mathcal{K}_\nu(\mathrm{spec} \, x, \underline{1}).$$
\end{remark}

\begin{theorem}\label{KBessel}
The Dunkl-type $\mathcal{K}$-Bessel function $\mathcal{K}_\nu$ exists for all $\nu \in \C$ and $w,z \in \C^n$ with $\Re \, w,\Re \, z >0$. Moreover, $\mathcal{K}$ is holomorphic in $\nu,w,z$ and has the following properties:
\begin{enumerate}[leftmargin=0.8cm, itemsep=5pt]
\item[\rm{(i)}] $\mathcal{K}_\nu(\sigma w,\sigma z)=\mathcal{K}_\nu(w, z)$ for all $\sigma \in \mathcal{S}_n$.
\item[\rm{(ii)}] $\mathcal{K}_\nu(w,z)=\mathcal{K}_{-\nu}(z,w)$ and $\abs{\mathcal{K}_\nu(w,z)}\le \mathcal{K}_{\Re \, \nu}(\Re \, w,\Re \, z)$.
\item[\rm{(iii)}] If $\nu \in \R$ and $x,y \in \R_+^n$, then
$$0<\mathcal{K}_\nu(x,y)\le \begin{cases}
\Gamma_n(\nu)\Delta(x)^{-\nu} & \te{ if } \nu>\mu_0 \\
\Gamma_n(-\nu)\Delta(y)^{\nu} & \te{ if } \nu<-\mu_0 
\end{cases}$$
and $\nu \mapsto \mathcal{K}_\nu(x,y)$ is convex.
\item[\rm{(iv)}] If $\nu \in \R$, $\abs{\nu}\le \mu_0$ and $x,y \in \R_+^n$, then for all $\epsilon>0$,
$$0<\mathcal{K}_\nu(x,y) \le \Gamma_n(\mu_0+1)\Delta(x)^{-\mu_0-1}+\Gamma_n(\mu_0+\epsilon)\Delta(y)^{-\mu_0-\epsilon}.$$
\item[\rm{(v)}] Recurrence formulas:
\begin{align*}
\Delta(T^A)\mathcal{K}_\nu(w,\m)&=(-1)^n\mathcal{K}_{\nu-1}(w,\m), \\
\Delta(T^A)\mathcal{K}_\nu(\m,z)&=(-1)^n\mathcal{K}_{\nu+1}(\m,z).
\end{align*}
\item[\rm{(vi)}] If $\Re \, \nu <-\mu_0$, then
$$\lim\limits_{\epsilon \to 0} \epsilon^{-n\nu}\mathcal{K}_\nu(w,\epsilon z) = \Gamma_n(-\nu)\Delta(z)^{\nu}.$$
If $\Re \, \nu >\mu_0$, then
$$\lim\limits_{\epsilon \to 0} \epsilon^{n\nu}\mathcal{K}_\nu(\epsilon w,z) = \Gamma_n(\nu)\Delta(w)^{-\nu}.$$
\end{enumerate}
These properties generalize those of $\mathcal{K}$-Bessel functions associated to symmetric cones as stated in \cite[Chapter XVI, Section 3]{FK94}.
\end{theorem}

The proof is similar to the case of symmetric cones as in \cite[Chapter XVI, Section 3]{FK94}, see also \cite{R06} for real positive definite symmetric matrices. 

\begin{proof}
Let $\Re \, \nu >\mu_0$ and $w,z \in \C^n$ with $\Re \, w,\Re \, z>0$. According to the estimates \eqref{DunklKernelEstimate},
$$\abs{E^A(-x,w)E^A(-\tfrac{1}{x},z)}\le E^A(-x,\Re \, w).$$
Hence, the Laplace transform identities in Theorem \ref{JackLaplace} lead to
\begin{align*}\label{InqualityKBessel}
\int_{\R_+^n}& \abs{E^A(-x,w)E^A(-\tfrac{1}{x},z) \Delta(x)^{\nu-\mu_0-1}}\omega^A(x) \dif x \\
&\le \int_{\R_+^n} E^A(-x,\Re \, w) \Delta(x)^{\Re \, \nu - \mu_0-1}\omega^A(x) \dif x 
= \Gamma_n(\Re \, \nu)\Delta(\Re \, w)^{-\Re\, \nu}.
\end{align*}
In particular, $\mathcal{K}_\nu(w,z)$ exists and is holomorphic in $(w,z,\nu) \in \set{\Re \, w>0}\times \set{\Re \, z>0} \times \set{\Re \, \nu >\mu_0}$ by usual theorems on holomorphic parameter integrals. Moreover, the stated estimate in (iii) is true in the case $\Re \, \nu>\mu_0$. \\
Consider the case $\Re\, \nu < -\mu_0$. By the change of variables $\xi \mapsto \tfrac{1}{\xi}$ in the defining integral for $\mathcal{K}_\nu(w,z)$ and Lemma \ref{InvariantMeasure} we have $\mathcal{K}_\nu(w,z)=\mathcal{K}_{-\nu}(z,w)$. Thus, $\mathcal{K}_\nu(w,z)$ exists and is holomorphic in $(w,z,\nu) \in \set{\Re \, w >0}\times\set{\Re \, z>0}\times \set{\Re\, \nu <-\mu_0}$. Thereby, the estimate in (iii) for $\Re \, \nu <-\mu_0$ is true as well. Finally, for $x \in \R_+^n$, $\nu \mapsto \Delta(x)^{\Re\,\nu-\mu_0-1}$ is convex on $\R$, so that $\mathcal{K}_\nu(w,z)$ exists and is holomorphic in $(w,z,\nu) \in \set{\Re\, w >0}\times\set{\Re \, z>0}\times \C$, where $\nu \mapsto \mathcal{K}_\nu(x,y)$ is convex on $\R$ for fixed $x,y \in \R_+^n$. It remains to prove the properties (i), (ii) and (iv)-(vi).
\begin{enumerate}[leftmargin=0.8cm, itemsep=5pt]
\item By properties of the Dunkl kernel, this follows by a change of variables.
\item By Lemma \ref{InvariantMeasure}, this is done by the change of variables $x \mapsto \tfrac{1}{x}$. The estimate is an immediate consequence of \eqref{DunklKernelEstimate}.
\item[(iv)] To prove the stated estimate, consider an arbitrary $\epsilon >0$ and split the defining integral for $\mathcal{K}_\nu$ according to
$$\R_+^n=\set{\Delta(\xi)<1} \sqcup \set{\Delta(\xi)>1} \sqcup \set{\Delta(\xi)=1},$$
where the last set has measure zero and need not to be discussed.
\begin{itemize}
\item[(a)] Since $E^A(-\tfrac{1}{\xi},y)\le 1$ for $\xi,y \in \R_+^n$ and $\Delta(\xi)^{\nu-\mu_0-1}<1$ for $\Delta(\xi)>1$, we obtain
\begin{align*}
\quad &\int_{\Delta(\xi)>1}E^A(-\xi,x)E^A(-\tfrac{1}{\xi},y)\Delta(\xi)^{\nu-\mu_0-1}\omega^A(\xi) \dif \xi \\
&\le \int_{\R_+^n} E^A(-\xi,x)\omega^A(\xi) \dif \xi = \Gamma_n(\mu_0+1)\Delta(x)^{-\mu_0-1},
\end{align*}
where again Theorem \ref{JackLaplace} was used.
\item[(b)] As $\Delta(\xi)^{-\nu-\mu_0-1}<\Delta(\xi)^{\epsilon-1}$ for $\Delta(\xi)>1$, we obtain
\begin{align*}
\quad &\int_{\Delta(\xi)<1}E^A(-\xi,x)E^A(-\tfrac{1}{\xi},y)\Delta(\xi)^{\nu-\mu_0-1}\omega^A(\xi) \dif \xi \\
&= \int_{\Delta(\xi)>1}E^A(-\tfrac{1}{\xi},x)E^A(-\xi,y)\Delta(\xi)^{-\nu-\mu_0-1}\omega^A(\xi) \dif \xi \\
&\le \int_{\R_+^n} E^A(-\xi,x)\Delta(\xi)^{\epsilon-1}\omega^A(\xi) \dif \xi = \Gamma_n(\mu_0+\epsilon)\Delta(x)^{-\mu_0-\epsilon}.
\end{align*}
\end{itemize}
\item[(v)] This is a consequence of the eigenvalue equation for $E^A$.
\item[(vi)] It suffices to check the second limit, the first one can then be deduced from (ii). Recall from \eqref{DunklKernelEstimate}, that $\abs{E^A(-\tfrac{\epsilon}{x},z)}\le 1$ for $\Re \, z>0$. By the change of variables $x \mapsto \tfrac{x}{\epsilon}$, dominated convergence and Theorem \ref{JackLaplace}, we just get
\begin{align*}
\quad \quad \quad \epsilon^{n\nu}\mathcal{K}_\nu(\epsilon w,z)&= \epsilon^{n\nu}\int_{\R_+^n} E^A(-x,\epsilon w)E^A(-\tfrac{1}{x},z)\Delta(x)^{\nu-\mu_0-1}\omega^A(x)\dif x \\
&= \int_{\R_+^n} E^A(-x,w)E^A(-\tfrac{\epsilon}{x},z)\Delta(x)^{\nu-\mu_0-1}\omega^A(x)\dif x,
\end{align*}
and as $\epsilon$ tends to $0$, the last integral converges to
$$\int_{\R_+^n} E^A(-x,w)\Delta(x)^{\nu-\mu_0-1} \omega^A(x)\dif x =\Gamma_n(\nu)\Delta(w)^{-\nu}.$$
\end{enumerate}
\end{proof}

\begin{proposition}
For $w \in \C^n$ with $\Re\, w>0$ put $f_w(x)=\mathcal{K}_{\nu-\mu_0-\tfrac{1}{2}}(x^2,w)$. Then, $f_w$ satisfies the eigenvalue equation
$$\frac{1}{4}(T_i^B)^2f_w=w_if_w \quad \te{ for all } i=1,\ldots,n,$$
where the type $B$ multiplicity of the Dunkl operator on the left side is given as before by the value $k$ on $\pm e_i \pm e_j$ and $k'=\nu-\mu_0-\tfrac{1}{2}$ on $\pm e_i$. 
\end{proposition}
This is a generalization of the Bessel system on a symmetric cone, which is solved by the $\mathcal{K}$-Bessel function, see for instance \cite[Page 358]{FK94} for the one-variable case, i.e. $w=\underline{1}$, or \cite{Mo13} for some further eigenvalue equations of the $\mathcal{K}$-Bessel function.

\begin{proof}
To improve readability we put $\nu'=\nu+\tfrac{1}{2}$. The proof will be done by direct computation and is divided into several steps:
\begin{enumerate}[leftmargin=0.8cm, itemsep=5pt]
\item First, consider a $C^2$-function $f(x)=f_0(x^2)$. Then, as noticed in \cite{BF98},
$$\hspace*{25pt}\frac{1}{4}(T_i^B)^2 f(x)\!=\!x_i^2 ((T_i^A)^2f_0)(x^2)+ (\nu'-\mu_0)(T_i^Af_0)(x^2)+k\sum\limits_{j \neq i} (T_i^Af_0)(s_{ij}x^2).$$
Hence, 
$$\hspace*{25pt}\frac{1}{4}(T_{i,x}^B)^2E^A(-x^2,\xi)\!=\!\left(x_i^2\xi_i^2-(\nu'-\mu_0)\xi_i\right)\!E^A(-x^2,\xi)-k\sum\limits_{j\neq i}\xi_iE^A(-x^2,s_{ij}\xi).$$
\item By a change of variables and $E^A(\sigma x,\sigma y)=E^A(x,y)$ for $\sigma \in \mathcal{S}_n$, we have
\begin{align*}
\hspace*{25pt}\int_{\R_+^n}\sum\limits_{j\neq i}\xi_i& E^A(-x^2,s_{ij}\xi) E^A(-\tfrac{1}{\xi},w)\Delta(\xi)^{\nu'-2(\mu_0+1)}\omega_k(\xi) \dif \xi \\
&=\int_{\R_+^n} E^A(-x^2,\xi) \left(\sum\limits_{j\neq i}\xi_j E^A(-\tfrac{1}{s_{ij}\xi},w)\right)\Delta(\xi)^{\nu'-2(\mu_0+1)}\omega_k(\xi) \dif \xi.
\end{align*}
\item For abbreviation put $g(\xi)=E^A(\xi,w)$, so that
\begin{align*}
\hspace*{35pt}T_{i,\xi}^A&(\xi_i^2g(-\tfrac{1}{\xi})) \\
&= \frac{\partial g}{\partial \xi_i}(-\tfrac{1}{\xi}) + 2\xi_ig(-\tfrac{1}{\xi}) + k \sum\limits_{j \neq i} \frac{\xi_i^2g(-\tfrac{1}{\xi})-\xi_j^2g(-\tfrac{1}{s_{ij}\xi})}{\xi_i-\xi_j} \\
&= (T_i^Ag)(-\tfrac{1}{\xi}) + 2\xi_ig(-\tfrac{1}{\xi_i}) + k\sum\limits_{j \neq i}\Bigg( \frac{\xi_i^2g(-\tfrac{1}{\xi})-\xi_j^2g(-\tfrac{1}{s_{ij}\xi})}{\xi_i-\xi_j} \\ 
&\quad -  \frac{g(-\tfrac{1}{\xi})-g(-s_{ij}\tfrac{1}{\xi})}{\tfrac{1}{\xi_j}-\tfrac{1}{\xi_i}} \Bigg) \\
&= (w_i+2\xi_i)g(-\tfrac{1}{\xi}) + k\sum\limits_{j \neq i} \xi_ig(-\tfrac{1}{\xi})+\xi_jg(-\tfrac{1}{s_{ij}\xi}) \\ 
&= (w_i+(2+\mu_0)\xi_i)g(-\tfrac{1}{\xi}) + k\sum\limits_{j \neq i}\xi_jg(-\tfrac{1}{s_{ij}\xi}).
\end{align*}
Moreover, since $T_i^A$ acts on $\mathcal{S}_n$-invariant functions as a partial derivative,
$$T_{i,\xi}^A\Delta(\xi)^{\nu'-2(\mu_0+1)}=(\nu'-2(\mu_0+1))\frac{\Delta(\xi)^{\nu'-2(\mu_0+1)}}{\xi_i}.$$
Therefore, 
\begin{align*}
\hspace*{35pt}&T_{i,\xi}^A\left(\xi_i^2E^A(-\tfrac{1}{\xi},w)\Delta(\xi)^{\nu'-2(\mu_0+1)}\right) \\
&= \Big[\left(w_i+(\nu'-\mu_0)\xi_i\right) E^A(-\tfrac{1}{\xi},w) + k\sum\limits_{j \neq i}\xi_j E^A(-\tfrac{1}{s_{ij}\xi},w)\Big]\Delta(\xi)^{\nu'-2(\mu_0+1)}.
\end{align*}
Thus, we conclude that
\begin{align*}
\hspace*{25pt}&\int_{\R_+^n}x_i^2\xi_i^2 E^A(-x^2,\xi) E^A(-\tfrac{1}{\xi},w)\Delta(\xi)^{\nu'-2(\mu_0+1)}\omega^A(\xi) \dif \xi \\
&=-\int_{\R_+^n} \xi_i^2T_{i,\xi}^AE^A(-x^2,\xi) E^A(-\tfrac{1}{\xi},w)\Delta(\xi)^{\nu'-2(\mu_0+1)}\omega^A(\xi) \dif \xi \\
&=\int_{\R_+^n}E^A(-x^2,\xi)  T_{i,\xi}^A\left( \xi_i^2E^A(-\tfrac{1}{\xi},w)\Delta(\xi)^{\nu'-2(\mu_0+1)}\right)\omega^A(\xi) \dif \xi \\
&=\int_{\R_+^n}\left(w_i+(\nu'-\mu_0)\xi_i\right) E^A(-x^2,\xi) E^A(-\tfrac{1}{\xi},w)\Delta(\xi)^{\nu'-2(\mu_0+1)}\omega^A(\xi) \dif \xi \\
&\quad + \int_{\R_+^n} E^A(-x^2,\xi)\m  k\sum\limits_{j \neq i}\xi_jE^A(-\tfrac{1}{s_{ij}\xi},w)\Delta(\xi)^{\nu'-2(\mu_0+1)}\omega^A(\xi) \dif \xi.
\end{align*}
Here, the second equation is justified by the skew symmetry of the Dunkl operators on $L^2(\R^n,\omega^A(x)\d x)$, which is based on integration by parts. So it suffices to show that there occur no boundary terms if we integrate by parts, which can be seen as follows: As $E^A(-\tfrac{1}{\xi},w) \longrightarrow 0$ for $\xi \to \partial \R_+^n$ and as $E^A(-x^2,\xi) \longrightarrow 0$ exponetially for $\xi \to \infty$ (estimate \ref{DunklKernelEstimate}), the boundary terms vanish on $\partial \R_+^n$ and in $\infty$.
\item Putting the things together, we obtain
\begin{align*}
\frac{1}{4}&(T_{i,x}^B)^2\mathcal{K}_{\nu'-\mu_0-1}(x^2,w) \\
&= \int_{\R_+^n} \frac{1}{4}(T_{i,x}^B)^2E^A(-x^2,\xi)E^A(-\tfrac{1}{\xi},w)\Delta(\xi)^{\nu'-2(\mu_0+1)}\omega(\xi) \dif \xi \\
&= w_i\mathcal{K}_{\nu'-\mu_0-1}(x^2,w).
\end{align*}
\end{enumerate}
\end{proof}

We may also study the $\mathcal{K}$-Bessel functions with a multivariate index, similar to those on symmetric cones. For this, we consider for $\lambda \in \C^n$ the unique analytic solution $\mathcal{G}(\lambda,\m)\coloneqq\mathcal{G}_k(\lambda,\m)$ in $\set{\Re\, z >0}$ of the following system
$$\begin{cases}
\mathscr{D}_j(k)f=\lambda_jf, \quad \te{for all } j=1,\ldots,n \\
f(\underline{1})=1
\end{cases}$$
with the so called Cherednik operators
$$\mathscr{D}_j(k)f(x)=\Big(x_jT_j^A(k)+k(1-n)+k\sum\limits_{i>j} s_{ij}\Big)f(x).$$
$\mathcal{G}$ is called the Cherednik kernel of type $A_{n-1}$. The reader is referred to \cite{BR22} where this kernel and its Dunkl-Laplace transform were studied. We summarize the following properties:

\begin{theorem}[{\cite{BR22}}]\label{CherednikKernel}
For arbitrary $\lambda,z \in \C^n$ with $\Re\, z>0$ and $\nu \in \C$, the Cherednik kernel $\mathcal{G}$ of type $A_{n-1}$ satisfies:
\begin{enumerate}[leftmargin=0.8cm, itemsep=5pt]
\item[\rm{(i)}] If $\lambda \in \Lambda_+^n$, then $\mathcal{G}(\lambda-\rho,z)=\frac{E_\lambda(z)}{E_\lambda(\underline{1})}$ is a normalized non-symmetric Jack polynomials of index $\alpha=\tfrac{1}{k}$ and $\rho=\rho(k)=-\frac{k}{2}(n-1,n-3,\ldots,-n+3,-n+1)$.
\item[\rm{(ii)}] $\mathcal{G}$ is positive on $\R^n\times \R_+^n$.
\item[\rm{(iii)}] $\mathcal{G}(\lambda+\underline{\nu},z)=\Delta(z)^\nu\mathcal{G}(\lambda,z)$ and $\mathcal{G}(-\rho,z)=1$.
\item[\rm{(iv)}] $\abs{\mathcal{G}(\lambda,x)}\le \mathcal{G}(\Re\, \lambda, x)$ and $\abs{\mathcal{G}(\lambda-\rho,x)} \le \max\limits_{\sigma \in \mathcal{S}_n} x^{\sigma \Re \, \lambda}$ for all $x \in \R_+^n$.
\item[\rm{(v)}] For $\lambda \in \C^n$ with $\Re \, \lambda >\mu_0$,
$$\int_{\R_+^n} E^A(-z,x)\mathcal{G}(\lambda-\rho,x)\Delta(x)^{-\mu_0-1}\omega^A(x) \dif x = \Gamma_n(\lambda)\mathcal{G}(\lambda,\tfrac{1}{x}).$$
\item[\rm{(vi)}] $\mathcal{G}(-\lambda,z)=\mathcal{G}(\lambda^R,\frac{1}{z^R})$ with $(y_1,\ldots,y_n)^R=(y_n,\ldots,y_1)$.
\end{enumerate}
\end{theorem}
The hypergeometric function of type $A_{n-1}$ is defined as
$$\mathcal{F}(\lambda,z)\coloneqq\mathcal{F}_k(\lambda,z)\coloneqq\frac{1}{n!}\sum_{\sigma \in \mathcal{S}_n}\mathcal{G}_k(\lambda,\sigma z).$$ In particular, for $\lambda \in \Lambda_+^n$, $\mathcal{F}(\lambda-\rho,z)=\frac{P_\lambda(z)}{P_\lambda(\underline{1})}$ is a normalized symmetric Jack polynomial of index $\alpha=\tfrac{1}{k}$. In the case $k=\tfrac{d}{2}$, where $d$ is the Peirce dimension constant of an irreducible symmetric cone, the $\mathcal{F}(\lambda,\m)$ are precisely the spherical functions of the cone. 

\begin{definition}
For arbitrary $\lambda \in \C^n$ we define the $\mathcal{K}$-Bessel function
$$\mathcal{K}_\lambda(w,z)=\int_{\R_+^n} E^A(-x,w)E^A(-\tfrac{1}{x},z)\mathcal{G}(\lambda-\rho,x)\Delta(x)^{-\mu_0-1}\omega^A(x) \dif x.$$
The convergence of the integral will be discussed in the next theorem and is in accordance with the previous case $\lambda=\underline{\nu}$ with $\nu \in \C$. This more general $\mathcal{K}$-Bessel function is also related to its counterpart on a symmetric cone $\Omega$. Similar computations as in Remark \ref{KBesselCone} show that
$$\int_{K\times K} \mathcal{K}_{\mathbf{s}}^\Omega(kx,k'y) \dif(k,k')=\frac{c_0}{n!^2}\sum\limits_{\tau,\sigma \in \mathcal{S}_n} \mathcal{K}_{\mathbf{s}}(\sigma\, \mathrm{spec}\, x, \tau\, \mathrm{spec}\, y).$$
One only has to use that
$$\int_K \Delta_{\mathbf{s}}(kx)\dif k = \mathcal{F}(\mathbf{s}-\rho,x)=\frac{1}{n!}\sum\limits_{\sigma \in \mathcal{S}_n}\mathcal{G}(\mathbf{s}-\rho,\sigma x),$$
which is, initially true for partitions $\mathbf{s} \in \Lambda_+^n$, since in this case, we have the spherical polynomials on the left hand side, which coincide with the Jack polynomials. By Carlson's theorem as in \cite[Lemma 5.2]{BR22}, the stated equality is valid for all $\mathbf{s} \in \C^n$. Similar continuation arguments are done in \cite{BR22}. 
\end{definition}

\begin{theorem}
The map $(\lambda,w,z)\mapsto \mathcal{K}_\lambda(w,z)$ is holomorphic on the domain $\C^n\times \set{\Re \, w >0} \times \set{\Re \, z>0}$ and satisfies:
\begin{enumerate}[leftmargin=0.8cm, itemsep=5pt]
\item[\rm{(i)}] $\mathcal{K}_\lambda(w,z)=\mathcal{K}_{-\lambda^R}(z^R,w^R)$ and $\abs{\mathcal{K}_\lambda(w,z)}\le \mathcal{K}_{\Re \, \lambda}(\Re \, w, \Re\, z)$.
\item[\rm{(ii)}] If $\lambda \in \R^n$ and $x,y \in \R_+^n$, then
\begin{align*}
0<\mathcal{K}_\lambda(x,y)& \le \begin{cases}
\Gamma_n(\lambda)\mathcal{G}(\lambda-\rho,\tfrac{1}{x}), & \te{if } \lambda > \mu_0 \\
\Gamma_n(-\lambda^R)\mathcal{G}(-\lambda^R-\rho,\tfrac{1}{x^R}), & \te{if } \lambda <\mu_0
\end{cases} \\
&=\begin{cases}
\Gamma_n(\lambda)\mathcal{G}(-\lambda^R-\rho,x^R), & \te{if } \lambda > \mu_0 \\
\Gamma_n(-\lambda^R)\mathcal{G}(\lambda-\rho,x), & \te{if } \lambda <\mu_0
\end{cases}.
\end{align*}
Moreover $\nu \mapsto \mathcal{K}_{\lambda+\underline{\nu}}$ is a convex function $\R \to \R_+$.
\item[\rm{(iii)}] Recurrence formulas:
\begin{align*}
\Delta(T^A)\mathcal{K}_\lambda(w,\m)&=(-1)^n\mathcal{K}_{\lambda-\underline{1}}(w,\m), \quad \Re \, w >0, \\
\Delta(T^A)\mathcal{K}_\lambda(\m,z)&=(-1)^n\mathcal{K}_{\lambda+\underline{1}}(\m,z), \quad \Re \, z >0.
\end{align*}
\end{enumerate}
\end{theorem}

\begin{proof}
The existence and analyticity of $\mathcal{K}_\nu(w,z)$ is checked for $\Re\, \lambda>\mu_0$ and $\Re\, \lambda<-\mu_0$ exactly as in Theorem \ref{KBessel}, using the stated properties of the Cherednik kernel in Theorem \ref{CherednikKernel}. The existence and analyticity in the case $-\mu_0\le \Re\, \lambda \le \mu_0$ can be deduced by a convexity property, as in Theorem \ref{KBessel}, which will be proven below. 
\begin{enumerate}[leftmargin=0.8cm, itemsep=5pt]
\item This is a immediate consequence of Theorem \ref{CherednikKernel} (vi).
\item Suppose that $\lambda > \mu_0$. Similar to the proof of Theorem \ref{KBessel} we obtain by Theorem \ref{CherednikKernel} (v),(vi) that
\begin{align*}
\hspace*{35pt}0<\mathcal{K}_\lambda(x,y)&\le \int_{\R_+^n}E^A(-\xi,x)\mathcal{G}(\lambda-\rho,x)\Delta(x)^{-\mu_0-1}\Delta(x)\omega^A(x) \dif x \\
&=\Gamma_n(\lambda)\mathcal{G}(\lambda-\rho,\tfrac{1}{x})=\Gamma_n(\lambda)\mathcal{G}(-\lambda^R-\rho,x^R).
\end{align*}
By part (i), the case $\lambda <- \mu_0$ reduces to the case $\lambda > \mu_0$. Moreover, in view of Remark \ref{CherednikKernel} (iii) we have 
$$\mathcal{G}(\lambda+\underline{\nu}-\rho(k),x)=\mathcal{G}(\lambda-\rho(k),x)\Delta(x)^\nu,$$ so that the convexity in $\nu$ is verified just as in Theorem \ref{KBessel}. Together with part (i), this shows that the integral exists for all $\lambda \in \C^n$ and has the stated analyticity property.
\item This is the same argument as in the proof of Theorem \ref{KBessel}.
\end{enumerate}
\end{proof}

\section{The Hankel transform and its connection to type $B$ Dunkl theory}\label{Hankel}
The Hankel transform for the root system $A_{n-1}$ was already introduced in \cite{BF98} and before in \cite{M13} in a symmetrized version, both at a rather formal level. In this section, we will discuss its analytic aspects and its connection to the type $B$ Dunkl transform.
The Bessel function $\mathcal{J}_\nu^\Omega$ of a symmetric cone $\Omega$ in a Euclidean Jordan algebra $V$ defines the Hankel transform on $\Omega$ by
$$\mathcal{H}^\Omega_\nu f(u)\coloneqq \frac{1}{\Gamma_\Omega(\nu)}\int_\Omega\mathcal{J}_\nu^\Omega(P(\sqrt{u})v) f(v)\det(v)^{\nu-\tfrac{m}{n}}\dif v,$$
cf. the notions from Section \ref{Bessels}. By Remark \ref{HypGeoSeries} and Definition \ref{Besselkernels}, the kernel $\mathcal{J}_\nu^\Omega(P(\sqrt{u})v)$ differs from $\mathscr{J}_\nu(\mathrm{spec}\, u ,\mathrm{spec}\, v)$ just by taking $K$-means and a arbitrary choice of the multiplicity $k$ in $\mathscr{J}_\nu$. In view of \cite[Theorem VI.2.3]{FK94}, the Hankel transform of a $K$-invariant function $f:V \to \C$, $f(u)=f_0(\mathrm{spec}\, u)$ becomes
$$\mathcal{H}_\nu^\Omega f(u)=\mathrm{const}\m \int_{\R^n} f_0(x)\mathscr{J}_\nu(x,\mathrm{spec}\, u) \Delta(x)^{\nu-\mu_0-1}\omega^A(x) \dif x.$$
This will serves as our starting point for the definition of a Dunkl-type Hankel transform. \\
Consider a self-adjoint representation $\phi:V \to \mathrm{End}\, E$ of $V$ on some real Euclidean space $E$, as well as the associated quadratic form $Q:E \to V$, whose image is contained in the closure $\overline{\Omega}$, see \cite[Chapter XVI]{FK94}. In \cite{H55, M82, R06} the following case was studied: the space $V=\mathrm{Sym}_n(\R)$ of real symmetric matrices, the cone $\Omega=\mathrm{Pos}_n(\R)$ of positive definite matrices, the Euclidean space $E=\R^{n\times m}$ equipped with the Hilbert-Schmidt inner product, the representation $\phi(x)\xi=x\xi$ and associated quadratic form $Q(\xi)=\xi\xi^T$.
It is important to note, that the Fourier transform of an integrable radial function $f(\xi)=F(Q(\xi))$ is given by a Hankel transform of $F$, namely
$$\widehat{f}(\eta)=\int_{E}e^{-i\braket{\xi,\eta}_E}f(\xi)d\xi=\te{const}\m\mathcal{H}^\Omega_\nu F(\tfrac{Q(\eta)}{4})$$
with $\nu=\tfrac{m}{2n}$. In the Dunkl setting, the Fourier transform is replaced by a Dunkl transform of type $B$, $Q$ is replaced by $x\mapsto x^2=(x_1^2,\ldots,x_n^2)$ and the parameter $\nu$ will be given by the associated multiplicity. \\
We proceed similar as in \cite{R06}, where the case $\Omega=\mathrm{Pos}_n(\R)$ is treated, for arbitrary symmetric cones see also \cite[Chapter XVI, Section 2]{FK94}. Some of our results are inspired by computations in \cite{M13}, which where done in the case of symmetric functions.

\begin{definition}
For $\nu \in \C$ with $\Re \, \nu >\mu_0$, the Hankel transform is defined as
$$\mathcal{H}_\nu f(w) \coloneqq \frac{1}{\Gamma_n(\nu)} \int_{\R_+^n} f(x)\mathscr{E}_\nu(x,w)\Delta(x)^{\nu-\mu_0-1} \omega^A(x) \dif x, \quad w \in \C^n,$$
whenever the integral exists for measurable $f:\R_+^n \to \C$. 
\end{definition}

\begin{lemma}\label{HankelProperty}
Consider $\nu \in \C$ with $\Re \, \nu > \mu_0$ and put $e_z(x)\coloneqq E^A(x,-z)$ for $z \in \C^n$ with $\Re \, z>0$. Then:
\begin{enumerate}[leftmargin=0.8cm, itemsep=5pt]
\item[\rm{(i)}] $e_z \in L_\nu^2(\R^n_+)\coloneqq L^2(\R_+^n,\Delta(x)^{\nu-\mu_0-1}\omega^A(x)\d x)$ for $\nu \in \R$ with $\nu>\mu_0$.
\item[\rm{(ii)}]  For $z \in \C^n$ with $\Re \, z >0$,
$$\mathcal{H}_\nu(E^A(\m,-z))=\Delta(z)^{-\nu}E^A(\m,-\tfrac{1}{z}).$$
In particular, $\mathcal{H}_\nu$ is involutive on $\mathcal{U}\coloneqq \mathrm{span}_\C\set{E^A(\m,-z)\mid \Re \, z>0}$.
\item[\rm{(iii)}]  For $\nu \in \R$, $\nu>\mu_0$ we may equip $L_\nu^2(\R^n_+)$ with the canonical inner product
$$\braket{f,g}_{L_\nu^2(\R^n_+)}\coloneqq \int_{\R_+^n} f(x)\overline{g(x)}\Delta(x)^{\nu-\mu_0-1}\omega^A(x)\dif x.$$
Then $\mathcal{U} \tm L_\nu^2(\R^n_+)$ is dense. More precisely, the span of each set of the form $\set{e_z\mid z=\underline{s}+iy, \, y \in \R^n}$ with fixed $s>0$ is already dense. 
\end{enumerate}
\end{lemma}

\begin{proof}
Part (i) is a consequence of estimate \eqref{DunklKernelEstimate} and part (ii) is a reformulation of Lemma \ref{BesselKernelLaplace}. Thus, it remains to prove part (iii). Fix some $s>0$ and assume that $f \in L_\nu^2(\R^n_+)$ satisfies $\braket{e_z,f}_{L_\nu^2(\R^n_+)}=0$ for all $z=\underline{s}+iy, \, y \in \R^n$. Then
$$0=\braket{e_{\underline{s}+iy},f}_{L_\nu^2(\R^n_+)}=\mathcal{L}(\overline{f}\Delta^{\nu-\mu_0-1})(\underline{s}+iy),$$
so that we conclude that $f\Delta^{\nu-\mu_0-1}=0$ a.e. by injectivity of the Dunkl-Laplace transform from Theorem \ref{LaplaceProperties} (iii). Hence $f=0$.
\end{proof}

\begin{theorem}\label{Unitary}
For $\nu \in \R$ with $\nu>\mu_0$, the Hankel transform $\mathcal{H}_\nu$ extends uniquely to an involutive isometric isomorphism of $L_\nu^2(\R^n_+)$. 
\end{theorem}

\begin{proof}
By Lemma \ref{HankelProperty} it suffices to show that $\mathcal{H}_\nu$ is unitary on the space generated by the functions $e_z(x)=E^A(x,-z)$ with $\Re \, z>0$. The Laplace transform of hypergeometric series shows that ${}_1K_0$ has an analytic extension (cf. \cite[Corollary 6.9]{BR22}) satisfying
$$\Gamma_n(\nu)\Delta(z)^{-\nu}{}_1K_0(\nu;w,-\tfrac{1}{z})=\int_{\R_+^n}E^A(-x,z)E^A(-x,w)\Delta(x)^{\nu-\mu_0-1}\omega^A(x) \dif x$$
for arbitrary $w \in \C^n$ and $z \in \C^n$ with $\Re\, w ,\, \Re \, z>0$.
As $\overline{e_w}=e_{\overline{w}}$ on $\R^n$
$$\braket{e_z,e_w}_{L_\nu^2(\R^n_+)}=\Gamma_n(\nu)\Delta(z)^{-\nu}{}_1K_0(\nu;\overline{w},-\tfrac{1}{z}).$$
Thus, by part (i) of Lemma \ref{HankelProperty} we conclude that
\begin{align*}
\braket{\mathcal{H}_\nu e_z,\mathcal{H}_\nu e_w}_{L_\nu^2(\R^n_+)}&=\Delta(z)^{-\nu}\Delta(\overline{w})^{-\nu}\braket{e_{1/z},e_{1/w}}_{L_\nu^2(\R^n_+)} \\
&=\Gamma_n(\nu)\Delta(\overline{w})^{-\nu}{}_1K_0(\nu;\tfrac{1}{\overline{w}},-z) \\
&=\overline{\Gamma_n(\nu)\Delta(w)^{-\nu}{}_1K_0(\nu;\tfrac{1}{w},-\overline{z})} \\
&=\overline{\braket{e_w,e_z}_{L_\nu^2(\R^n_+)}}= \braket{e_z,e_w}_{L_\nu^2(\R^n_+)}.
\end{align*}
\end{proof}

The Hankel transform $\mathcal{H}_\nu f$ of a function $f \in L_\nu^2(\R^n_+)$ can also be uniquely characterized in terms of the Dunkl-Laplace transform with the following lemma.
\begin{lemma}\label{LaplaceHankel}
Consider $f \in L_\nu^2(\R^n_+)$, so in particular $\mathcal{L}_k(f\Delta^{\nu-\mu_0-1})(z)$ exists for all $z \in \C^n$ with $\Re \, z >0$. Moreover, assume that $g:\R_+^n \to \C$ is measurable and $s >0$, such that $\mathcal{L}(g\Delta^{\nu-\mu_0-1})(\underline{s})$ exists. Then:
\begin{enumerate}[leftmargin=0.8cm, itemsep=5pt]
\item[\rm{(i)}] For all $z \in \C^n$ with $\Re \, z>0$
$$\mathcal{L}((\mathcal{H}_\nu f)\Delta^{\nu-\mu_0-1})(z) = \Delta(z)^{-\nu}\mathcal{L}(f\Delta^{\nu-\mu_0-1})(\tfrac{1}{z}).$$
\item[\rm{(ii)}] If $\mathcal{L}(g\Delta^{\nu-\mu_0-1})(z)=\Delta(z)^{-\nu}\mathcal{L}(f\Delta^{\nu-\mu_0-1})(\tfrac{1}{z})$ for all $z \in \C$ with $\Re\, z>s$, then $g \in L_\nu^2(\R^n_+)$, $\mathcal{L}(g\Delta^{\nu-\mu_0-1})(z)$ exists for all $\Re \, z>0$ and $g=\mathcal{H}_\nu f$.
\end{enumerate}
\end{lemma}

\begin{proof}
For $g=\mathcal{H}_\nu f$ we have by Lemma \ref{HankelProperty} and Theorem \ref{Unitary}
\begin{align*}
\mathcal{L}(g\Delta^{\nu-\mu_0-1})(z)&=\braket{e_z,\overline{g}}_{L_\nu^2(\R^n_+)} =\braket{\mathcal{H}_\nu e_z,\overline{f}}_{L_\nu^2(\R^n_+)} \\
&=\Delta(z)^{-\nu} \braket{e_{1/z},\overline{f}}_{L_\nu^2(\R^n_+)} = \Delta(z)^{-\nu}\mathcal{L}(f\Delta^{\nu-\mu_0-1})(\tfrac{1}{z}).
\end{align*}
This proves (i). The second property follows immediately from injectivity of the Dunkl-Laplace transform and part (i), since from the assumption we conclude as before $\mathcal{L}(g\Delta^{\nu-\mu_0-1})(z)=\mathcal{L}((\mathcal{H}_\nu f) \Delta^{\nu-\mu_0-1})(z)$ for all $z\in \C^n$ with $\Re \, z>s$.
\end{proof}

To facilitate readability, we will write $L^p(\Omega,h(x)\d x)\coloneqq L^p(\Omega,\abs{h(x)}\d x)$ for measurable $h:\Omega \to \C$, where $\Omega \tm \R^n$ is a Borel set. \\ 
Recall the correspondence between the pairs of type $A$ multiplicities $\kappa_A=k$ together with a parameter $\nu \in \C$, and type $B$ multiplicities $\kappa_B$, namely
$$(\kappa_A,\nu)=(k,\nu) \longleftrightarrow \kappa_B=(k,k') \te{ with } k'=\nu-\mu_0-\tfrac{1}{2},\, \mu_0=k(n-1).$$
Further, assume that $\Re \, \nu > \mu_0+\tfrac{1}{2}$, i.e. $\Re\, \kappa_B \ge 0$, so that $\mathcal{F}^B$ is an automorphism of $\mathscr{S}(\R^n)$, injective on $L^1(\R^n,\omega^B(x)\d x)$, and extends to an unitary map of $L^2(\R^n,\omega^B(x)\d x)$ in the case $\kappa_B \ge 0$. \\
The following integral decomposition has to be seen as a Dunkl analogue of the formula \cite[Proposition XVI.2.1]{FK94}, where we replaced the integration over the Stiefel manifold by summation over the $\Z_2^n$-action. This is quite simple, but the stated formula has the same important role as on symmetric cones.

\begin{proposition}[Integral decomposition]\label{IntegralDecomposition}
Consider $f\in L^1(\R^n, \omega^B(x)\d x)$. Then
$$\int_{\R^n}f(x)\omega^B(x)\dif x = \frac{1}{2^n}\sum\limits_{\tau \in \Z_2^n} \int_{\R_+^n}f(\tau x^{\tfrac{1}{2}})\Delta(x)^{\nu-\mu_0-1}\omega^A(x) \dif x,$$
where $x^{\tfrac{1}{2}}$ has to be understood componentwise. In particular, if $f$ is $\Z_2^n$-invariant
\begin{equation}\label{B-invariant}
\int_{\R^n}f(x)\omega^B(x)\dif x = \int_{\R_+^n}f(x^{\tfrac{1}{2}})\Delta(x)^{\nu-\mu_0-1}\omega^A(x) \dif x, 
\end{equation}
\end{proposition}

\begin{proof}
Since $\omega^B$ is $\Z_2^n$-invariant, we have with \eqref{WeightRelation} that
$$\int_{\R^n}f(x)\omega^B(x) \dif x = \sum\limits_{\tau \in \Z_2^n}\int_{\R_+^n} f(\tau x)\Delta(x^2)^{k'}\omega^A(x^2) \dif x.$$
Hence, the change of variables $x \leftrightarrow x^{\tfrac{1}{2}}$ and $\nu-\mu_0-1=k'-\tfrac{1}{2}$ gives the stated formula.
\end{proof}

The Hankel transform and the type $B$ Dunkl transform are closely related. The connection is given in the subsequent theorem.

\begin{theorem}\label{HankelDunkl}
Recall the Dunkl transform $\mathcal{F}^B$. Then $c_B=2^{n\nu}\Gamma_n(\nu)$ and ff $f \in L^1(\R^n,\omega^B(x)\d x)$ is $\Z_2^n$-invariant, the measurable function $f_0:\R_+^n \to \C$ defined by $f(x)=f_0(x^2)$ satisfies
\begin{equation}\label{HankelDunklEq}
\mathcal{F}^Bf (\xi)=2^{-n\nu}\mathcal{H}_\nu f_0 (\tfrac{\xi^2}{4}), \quad \xi \in \R^n.
\end{equation}
\end{theorem}

\begin{proof}
The computation of $c_B$ can be deduced from Theorem \ref{JackLaplace} and Proposition \ref{IntegralDecomposition}
\begin{align*}
c_B&=\int_{\R^n}e^{-\tfrac{\abs{x}^2}{2}}\omega^B(x) \dif x = \int_{\R_+^n} e^{-\tfrac{1}{2}\braket{x,\underline{1}}}\Delta(x)^{\nu-\mu_0-1}\omega^A(x) \dif x \\
&= 2^{n\nu}\int_{\R_+^n} E^A(-x,\underline{1})\Delta(x)^{\nu-\mu_0-1}\omega^A(x) \dif x = 2^{n\nu}\Gamma_n(\nu).
\end{align*}
Further, the integral formula \eqref{B-invariant} of Proposition \ref{IntegralDecomposition} and the relation between the Bessel kernel $\mathscr{E}_\nu$ with squared argument and the type $B$ Dunkl kernel from Proposition \ref{BesselKernelTypeBConnection}, we deduce that
$$\mathcal{F}^Bf(\xi)=\frac{1}{c_B}\int_{\R_+^n} f_0(x)\mathscr{E}_\nu(x,\tfrac{\xi^2}{4})\Delta(x)^{\nu-\mu_0-1}\omega^A(x)\dif x,$$
so \eqref{HankelDunklEq} is proven.
\end{proof}

The following lemma is a Dunkl analogue of \cite[Theorem 3.1]{R06}. On the space $\R^{n\times m}$ of rectangular matrices the author describes the action of the so-called Cayley-Laplacian $\det(\partial^T \partial)$ on radial functions, i.e. functions of the form $f(x)=f_0(xx^T)$. We replace the Cayley-Laplacian the type $B$ Dunkl operator $\Delta(T^B)^2$ and describe its action on $\Z_2^n$-invariant functions $f$.

\begin{theorem}\label{CayleyLaplaceEven}
Let $f \in C^\infty(\R^n)$ be a $\Z_2^n$-invariant function.
\begin{enumerate}[leftmargin=0.8cm, itemsep=5pt]
\item[\rm{(i)}] There exists a smooth function $f_0\in C^\infty(\R^n)$ with $f(x)=f_0(x^2)$.
\item[\rm{(ii)}] The Dunkl operator $\Delta(T^B)^2$ acts according to
\begin{equation}\label{CayleyLaplace}
\Delta(T^B)^2f(x)=\mathscr{L}_\nu f_0(x^2),
\end{equation}
where the operator on the right hand side is defined as
$$\mathscr{L}_\nu\coloneqq 4^n\Delta(x)^{1+\mu_0-\nu}\Delta(T^A)\Delta(x)^{\nu-\mu_0}\Delta(T^A).$$
Here the powers of $\Delta(x)$ are understood as multiplication operator.
\end{enumerate}
\end{theorem}

\begin{proof}
\
\begin{enumerate}[leftmargin=0.8cm, itemsep=5pt]
\item This was done by Whitney \cite{W43} for univariate functions, and the multivariate case is easily reduced to this case.
\item By continuity it suffices to check \eqref{CayleyLaplace} on $\R^n \setminus \set{\Delta(x)=0}$. Since
$$\supp( T_\xi^R f) \tm W_{\! R}. \supp \,f \quad \te{ for } R\in \set{A,B}$$
we can assume without loss of generality that $f \in C_c^\infty(\R^n)$ and that $\supp \,f \cap \set{\Delta(x)=0}=\emptyset$, i.e. that $f_0 \in C_c^\infty(\R_+^n)$. Since the Dunkl transform $\mathcal{F}^B$ is  injective, we can prove identity \eqref{CayleyLaplace} under the action of $\mathcal{F}^B$. Identities \eqref{DunklTrafo} and Theorem \ref{HankelDunkl} show that
\begin{align*}
\quad \quad \quad \quad \mathcal{F}^B(\Delta(T^B)^2&f)(\xi)=\Delta(i\xi)^2 \mathcal{F}^Bf(\xi) \\
&= \frac{(-1)^n2^{-n\nu}}{\Gamma_n(\nu)}\Delta(\xi)^2\int_{\R_+^n} \mathscr{E}_\nu(\tfrac{\xi^2}{4},x)f_0(x)\Delta(x)^{\nu-\mu_0-1}\omega^A(x) \dif x.
\end{align*}
The recurrence formulas of Corollary \ref{RecurrenceBessel} lead to
\begin{align*}
\quad \quad \quad \quad (-1)^n\Delta(\xi)^2\mathscr{E}_\nu(\tfrac{\xi^2}{4},x)\Delta(x)^{\nu-\mu_0-1} &= (-4)^n \Delta(\tfrac{\xi^2}{4})\mathscr{E}_\nu(\tfrac{\xi^2}{4},x)\Delta(x)^{\nu-\mu_0-1} \\
&= 4^n\Delta(T^A_{x})\!\left(\!\Delta(x)^{\nu-\mu_0}\Delta(T^A_{x})\mathscr{E}_\nu(\tfrac{\xi^2}{4},x)\!\right).
\end{align*}
The Dunkl operators $T_\xi^A$ are skew symmetric on $C_c^\infty(\R^n) \tm L^2(\R^n, \omega^A(x)\d x)$ (cf. \cite{dJ93}). Hence we obtain for $g(x)=\mathscr{L}_\nu f_0(x^2)$ that
\begin{align*}
\quad \quad \quad \mathcal{F}^B&(\Delta(T^B)^2f)(\xi) \\
&= \frac{4^n2^{-n\nu}}{\Gamma_n(\nu)} \int_{\R_+^n} \big(\Delta(T^A_{x})\big(\Delta(x)^{\nu-\mu_0-1}\Delta(T^A_{x})\mathscr{E}_\nu(\tfrac{\xi^2}{4},x)\big)\big) f_0(x) \omega^A(x) \dif x \\
&= \frac{4^n2^{-n\nu}}{\Gamma_n(\nu)} \int_{\R_+^n}\mathscr{E}_\nu(\tfrac{\xi^2}{4},x) (\Delta(T^A_{x})(\Delta(x)^{\nu-\mu_0-1}\Delta(T^A_{x})f_0))(x)\omega^A(x) \dif x \\
&= \frac{2^{-n\nu}}{\Gamma_n(\nu)} \int_{\R_+^n}\mathscr{E}_\nu(\tfrac{\xi^2}{4},x) (\mathscr{L}_\nu f_0)(x) \Delta(x)^{\nu-\mu_0-1}\omega^A(x) \dif x \\
&=2^{-n\nu}\mathcal{H}_\nu(\mathscr{L}_\nu f_0)(\tfrac{\xi^2}{4}) = \mathcal{F}^Bg(\xi),
\end{align*}
where in the last line again Theorem \ref{HankelDunkl} was used.
\end{enumerate}
\end{proof}

\begin{example}[Type $B$ Bernstein identity]\label{Bernstein}
Due to \cite[Lemma 5.4]{R20} or equally by Theorem \ref{JackLaplace} for $\eta=\underline{1}$ one has the following type $A$ Bernstein identity:
$$\Delta(T^A)\Delta(x)^\mu=b(\mu)\Delta(x)^{\mu-1}, \quad \mu \in \C,$$
where $b(\mu)=\prod_{j=1}^n(\mu+k(j-1))$. If we choose $f(x)=\Delta(x^2)^\mu$ in Theorem \ref{CayleyLaplaceEven}, we obtain from this type $A$ Bernstein identity that
\begin{align*}
\Delta(T^B)^2\Delta(x^2)^\mu &= 4^n [\Delta(\xi)^{1+\mu_0-\nu}\Delta(T^A)\Delta(\xi)^{\nu-\mu_0}\Delta(T^A)\Delta(\xi)^\mu]_{\xi=x^2} \\
&= 4^n b(\mu) [\Delta(\xi)^{1+\mu_0-\nu}\Delta(T^A)\Delta(\xi)^{\nu+\mu-\mu_0-1}]_{\xi=x^2} \\
&= 4^nb(\mu)b(\nu+\mu-\mu_0-1)\Delta(x)^{\mu-1}\\
&\eqqcolon \mathcal{B}(\mu)\Delta(x^2)^{\mu-1}.
\end{align*}
Thus, we have
\begin{equation}\label{BBernstein}
\Delta(T^B)^2\Delta(x^2)^\mu = \mathcal{B}(\mu)\Delta(x^2)^{\mu-1}, 
\end{equation}
where $\mathcal{B}$ is the polynomial $\mathcal{B}(\mu)=4^n\prod_{j=1}^n(\mu+k(j-1))(\mu-\tfrac{1}{2}+k'+k(j-1))$. \\
This Bernstein identity was independently proven in \cite[Proposition 3.1.2]{L16} by direct computation.
\end{example}

\section{Zeta integrals and zeta distributions in the type B Dunkl setting}
Recall the symmetric cone setting from the previous section. For $\alpha \in \C$ with $\Re \, \alpha > \tfrac{d}{2}(n-1)-\tfrac{m}{2n}$, the zeta integral of index $\alpha$ is defined a Schwartz function $f \in \mathscr{S}(E)$, $E=\R^{n\times m}$ by
$$Z(f;\alpha)=\int_E f(\xi) \det(Q(\xi))^\alpha \dif \xi,$$
where $\det$ is the Jordan determinant, see \cite[Chapter XVI]{FK94}. The zeta integral can be meromorphically extended in the parameter $\alpha$ to the complex plane $\C$ such that it satisfies the following characteristic functional equation
$$\frac{Z(\widehat{f};\alpha-\tfrac{m}{2n})}{\Gamma_\Omega(\alpha)}=\pi^{\tfrac{n}{2}}4^{n\alpha}\frac{Z(f;-\alpha)}{\Gamma_\Omega(\tfrac{m}{2n}-\alpha)},$$
where $\Gamma_\Omega$ is again the Gindikin gamma function of $\Omega$. We shall derive a similar functional equation, where the Fourier transform on $E$ is replaced by the type $B$ Dunkl transform on $\R^n$ and the dimension constant $\frac{m}{2n}$ is replaced by a continuous parameter depending on the type $B$ multiplicity.

To become more precise, consider $\kappa_A=k\ge 0$, $\nu \in \C$ with $\Re\, \nu > \mu_0+\tfrac{1}{2}$ and $\kappa_B=(k,k')$ with $k'=\nu-\mu_0-\tfrac{1}{2}$. 
\begin{definition}
We define the zeta integral of index $\alpha$, $\Re \, \alpha>\mu_0$ of a Schwartz function $f \in \mathscr{S}(\R^n)$ by
$$\mathcal{Z}(f;\alpha)\coloneqq \int_{\R^n} f(x)\Delta(x^2)^{\alpha-\nu}\omega^B(x) \dif x = \int_{\R^n} f(x)\Delta(x^2)^{\alpha-\mu_0-1}\omega^A(x^2) \dif x.$$
Obviously, the integral converges absolutely and depends holomorphically on $\alpha$. Moreover, $\mathcal{Z}(\m,\alpha)$ is $W_{\! B}$-invariant, i.e. $\mathcal{Z}(\tau f,\alpha)=\mathcal{Z}(f,\alpha)$ for all $\tau \in W_{\! B}$.
\end{definition}

The following example will show which kind of functional equation we can expect for this zeta integral in the Dunkl setting and how it can be extended meromorphically in the parameter $\alpha \in \C$.

\begin{example}\label{GaussFunction}
Consider the Gauss function
$$g(x)\coloneqq e^{-\abs{x}^2}=E^A(-\underline{1},x^2).$$
If $\alpha \in \C$ with $\Re \, \alpha > \mu_0$, then Theorem \ref{JackLaplace} and Proposition \ref{IntegralDecomposition} show that
$$\mathcal{Z}(g;\alpha)=\int_{\R_+^n} E^A(-\underline{1},x)\Delta(x)^{\alpha-\mu_0-1}\omega^A(x) \dif x = \Gamma_n(\alpha).$$
Therefore, the zeta integral of $g$ can be meromorphically extended in the parameter $\alpha \in \C$. Moreover, the type $B$ Dunkl transform of $g$ can be computed by Lemma \ref{BesselKernelLaplace} and Theorem \ref{HankelDunkl} as
$$\mathcal{F}^Bg(x)=2^{-n\nu}g(\tfrac{x}{2})=2^{-n\nu}E^A(-\underline{1},\tfrac{x^2}{4}).$$
Thus, for $\Re \, \alpha <\Re \, \nu-\mu_0$, we get
\begin{align*}
\mathcal{Z}(\mathcal{F}^Bg;\nu-\alpha) &= 2^{-n\nu}\int_{\R_+^n}E^A(-\underline{1},\tfrac{x}{4})\Delta(x)^{\nu-\alpha-\mu_0-1}\omega^A(x) \dif x \\
&= 2^{n(\nu-2\alpha)}\int_{\R_+^n}e^{-\braket{x,\underline{1}}}\Delta(x)^{\nu-\alpha-\mu_0-1}\omega^A(x) \dif x \\
&=2^{n(\nu-2\alpha)}\Gamma_n(\nu-\alpha),
\end{align*}
which can be extended meromorphically in $\alpha \in \C$. We therefore obtain
\begin{equation}\label{GaussZeta}
\frac{\mathcal{Z}(g,\alpha)}{\Gamma_n(\alpha)}=2^{n(2\alpha-\nu)}\frac{\mathcal{Z}(\mathcal{F}^Bg,\nu-\alpha)}{\Gamma_n(\nu-\alpha)} = 1.
\end{equation}
We see that both sides of this equation are entire functions in $\alpha$ and we have found a functional equation between the zeta integrals of $g$ and $\mathcal{F}^Bg$. We shall see that both the functional equation and the analytic extension will be valid for arbitrary Schwartz functions.
\end{example}

\begin{proposition}\label{EvenSchwartz}
Consider a $\Z_2^n$-invariant Schwartz function $f\in\mathscr{S}(\R^n)$. Then there exists a Schwartz function $f_0 \in \mathscr{S}(\R^n)$ such that $f(x)=f_0(x^2)$.
\end{proposition}

\begin{proof}
Since $f|_{\R_+^n}$ extends to a Schwartz function, we can use \cite[Theorem 1.1, $\eta=0$]{S19} to obtain for arbitrary $\alpha,\beta \in \N_0^n$ that
$$ \sup\limits_{x \in \R^n_+}\abs{x^\alpha\left(\tfrac{1}{x}\tfrac{\partial}{\partial x}\right)^\beta f(x)}<\infty,$$
where $\left(\tfrac{1}{x}\tfrac{\partial}{\partial x}\right)^\beta= \left(\tfrac{1}{x_1}\tfrac{\partial}{\partial x_1}\right)^{\beta_1}\cdots \left(\tfrac{1}{x_n}\tfrac{\partial}{\partial x_n}\right)^{\beta_n} .$
As stated in Theorem \ref{CayleyLaplaceEven} (i), we can find a smooth function $g \in C^\infty(\R^n)$ with $f(x)=g(x^2)$. Consider the change of variables $t_i=x_i^2$, i.e. $\tfrac{1}{x_i}\tfrac{\partial}{\partial x_i}=2\tfrac{\partial}{\partial t_i}$. The function $g$ satisfies 
\begin{align*}
\sup\limits_{t \in \R_+^n} \abs{t^\alpha \left( \tfrac{\partial}{\partial t}\right)^\beta\!\! g(t)} &= \frac{1}{2^{\abs{\beta}}}\sup\limits_{x \in \R_+^n} \abs{x^{2\alpha} \left[\left( \tfrac{1}{x}\tfrac{\partial}{\partial x}\right)^\beta\!\! f\right](x)} <\infty.
\end{align*} 
Application of \cite[Theorem 4.2]{JP16} now shows that the function $g|_{\R_+^n}$ is the restriction of a Schwartz function $f_0 \in \mathscr{S}(\R^n)$. In particular, $f(x)=f_0(x^2)$.
\end{proof}

Our zeta integrals have a close connection to the Riesz distributions in the Dunkl setting of type $A$, which are defined for $\alpha \in \C, \, \Re \, \alpha >\mu_0$ by
$$\braket{R_\alpha, f}\coloneqq \frac{1}{\Gamma_n(\alpha)}\int_{\R_+^n} f(x)\Delta(x)^{\alpha-\mu_0-1} \omega^A(x) \dif x, \quad f \in \mathscr{S}(\R^n).$$
These Riesz distributions were studied in \cite{R20}. They extend to a (weakly) holomorphic map $\C \to \mathscr{S}'(\R^n), \, \alpha \mapsto R_\alpha$.

\begin{lemma}\label{RieszZeta}
Let $f \in \mathscr{S}(\R^n)$. Then the function
$$\alpha \mapsto \braket{\zeta_\alpha,f}\coloneqq \frac{\mathcal{Z}(f;\alpha)}{\Gamma_n(\alpha)}$$
extends to an entire function on $\C$. If $f \in \mathscr{S}(\R^n)$ is $\Z_2^n$-invariant with $f(x)=f_0(x^2)$ and $f_0 \in\mathscr{S}(\R^n)$, then this extension is given in terms of the Riesz distributions $R_\alpha$ via
$$\braket{\zeta_\alpha,f}=\braket{R_\alpha,f_0}.$$
\end{lemma}

We call $\zeta_\alpha$ a Dunkl-type zeta distribution of index $\alpha \in \C$. That $\zeta_\alpha$ is in fact a tempered distribution will be proven in the subsequent theorem.

\begin{proof}
By $\Z_2^n$-invariance of the zeta integrals, we may asssume that $f$ is a $\Z_2^n$-invariant function, otherwise we can consider $\Z_2^n$-mean. By Theorem \ref{EvenSchwartz} we can find $f_0 \in \mathscr{S}(\R^n)$ such that $f(x)=f_0(x^2)$. For $\alpha \in \C$ with $\Re \, \alpha >\mu_0$, Proposition \ref{IntegralDecomposition} shows that
$$\braket{\zeta_\alpha,f}=\braket{R_\alpha,f_0}.$$
Hence, that statement follows from the analytic extension property for the Dunkl-type Riesz distributions $R_\alpha,\alpha \in \C$.
\end{proof}

Recall the Dunkl operators $T_\xi^R$ and the Dunkl transform $\mathcal{F}^R$ act continuously on the space of Schwartz functions $\mathscr{S}(\R^n)$, equipped with the usual locally convex topology. Hence, by duality we consider the usual actions on the space of tempered distributions $\mathscr{S}'(\R^n)$ via
\begin{align*}
\braket{T_\xi^Ru, \m} &= \braket{u,-T_\xi^R \m}, \\
\braket{\mathcal{F}^Ru,\m} &=\braket{u,\mathcal{F}^R\m}, 
\end{align*}
for $u \in \mathscr{S}'(\R^n)$.

\begin{theorem}\label{ZetaProperties}
The functionals $\zeta_\alpha$ with $\alpha \in \C$ have the following properties:
\begin{enumerate}[leftmargin=0.8cm, itemsep=5pt]
\item[\rm{(i)}] $\displaystyle \zeta_\alpha \in \mathscr{S}'(\R^n)$ and $\zeta_\alpha$ is $W_{\! B}$-invariant. 
\item[\rm{(ii)}] If $\Re \, \alpha >\mu_0$, then $\zeta_\alpha$ is a positive measure with support $\R^n$.
\item[\rm{(iii)}] $\displaystyle \Delta(T^B)^2\zeta_\alpha = 4^nb(\alpha-\nu)\zeta_{\alpha-1}$ with $b(\mu)=\prod\limits_{j=1}^n(z+k(j-1))$.
\item[\rm{(iv)}] $\displaystyle \Delta(x^2)\zeta_\alpha = b(\alpha-\mu_0)\zeta_{\alpha+1}$.
\end{enumerate}
\end{theorem}

\begin{proof}
Part (ii) is obvious. If $\Re \, \alpha >\mu_0+1$, then the skew symmetry of Dunkl operators and the Bernstein identity \eqref{BBernstein} of Example \ref{Bernstein} show that
$$\mathcal{Z}(\Delta(T^B)^2f;\alpha)=\mathcal{B}(\alpha-\nu)\mathcal{Z}(f;\alpha-1).$$
In particular,
$$\braket{\zeta_\alpha, \Delta(T^B)^2f}=\frac{\Gamma_n(\alpha-1)}{\Gamma_n(\alpha)}\mathcal{B}(\alpha-\nu)\braket{\zeta_{\alpha-1},f}.$$
By definition, $\mathcal{B}(\alpha-\nu)=4^n b(\alpha-\nu)b(\alpha-\mu_0-1)$ and moreover
$$b(\alpha-\mu_0-1)=\prod\limits_{j=1}^n (\alpha-1-k(j-1))=\frac{\Gamma_n(\alpha)}{\Gamma_n(\alpha-1)}.$$
Hence, 
$$\braket{\zeta_\alpha,\Delta(T^B)^2f}=4^nb(\alpha-\nu)\braket{\zeta_{\alpha-1},f},$$ 
so that analytic extension according to Lemma \ref{RieszZeta} shows that $\zeta_{\alpha-1}$ is a tempered distribution provided $\zeta_\alpha$ is so and $b(\alpha-\nu) \neq 0$. The set 
$$M=\set{\alpha \in \C \mid b(\alpha+k-\nu) \neq 0 \te{ for all } k\in \Z}$$ is dense in $\C$ and for all $\alpha \in M$ we have $\zeta_\alpha \in \mathscr{S}'(\R^n)$. But, the map $\alpha \mapsto \braket{\zeta_\alpha,f}$ is holomorphic for all $f \in \mathscr{S}(\R^n)$ and as a dual of a Fréchet space, it is closed under pointwise limits. Hence all $\zeta_\alpha$ are tempered. \\
Thus, we have proven parts (i) and (iii). Part (iv) is immediate for $\Re \, \alpha >\mu_0$, since
$$b(\alpha-\mu_0)=\frac{\Gamma_n(\alpha+1)}{\Gamma_n(\alpha)},$$
and follows by analytic extension in general.
\end{proof}

In line with Example \ref{GaussFunction} we next obtain a general functional equation for our zeta distributions. The idea of the proof is the same as for the analogous results in \cite{FK94, R06}. The $\mathcal{K}$-Bessel function and its asymptotic properties will play an essential role.

\begin{theorem}\label{FunctionalEquation}
For all $\alpha \in \C$, the zeta distributions $\zeta_\alpha$ satisfy the functional equation
\begin{equation}\label{EqZeta}
\zeta_\alpha = 2^{n(2\alpha-\nu)}\mathcal{F}^B\zeta_{\nu-\alpha}.
\end{equation}
Moreover, in view of by Theorem \ref{HankelDunkl},
$$\braket{R_\alpha,f} = 4^{n(\alpha-\nu)}\braket{R_{\nu-\alpha},\mathcal{H}_\nu g}, \quad f \in \mathscr{S}(\R^n), g(x)=f(\tfrac{x}{4}).$$
\end{theorem}

For the proof of identity \eqref{EqZeta}, we start with two lemmata, which are needed to outsource some technicalities.

\begin{lemma}\label{ZFLemma1}
Consider $g \in \mathscr{S}(\R^n)$ and $\alpha \in \C$.
\begin{enumerate}[leftmargin=0.8cm, itemsep=5pt]
\item[\rm{(i)}] Assume that $\Re \, \alpha >\mu_0$. Then for arbitrary $\epsilon>0$, the following integrals exist and coincide
\begin{align*}
&\int_{\R_+^n} \int_{\R^n} g(x)E^A(-\tfrac{1}{s},x^2+\underline{\epsilon})\Delta(s)^{-\alpha-\mu_0-1}\omega^B(x)\omega^A(s) \dif x \dif s \\
&=\Gamma_n(\alpha)\int_{\R^n}g(x)\Delta(x^2+\underline{\epsilon})^{-\alpha}\omega^B(x)\dif x.
\end{align*}
\item[\rm{(ii)}] For $\Re \, \alpha < \Re \, \nu- \mu_0$ and $\epsilon >0$, the following integral and limit exists
$$\lim\limits_{\epsilon \to 0}\int_{\R^n}g(x)\Delta(x^2+\underline{\epsilon})^{-\alpha}\omega^B(x)\dif x =\mathcal{Z}(g,\nu-\alpha).$$ 
Furthermore, the integral on the left hand side exists for arbitrary $\alpha \in \C$ and defines an entire function in $\alpha$.
\end{enumerate}
\end{lemma}

\newpage 

\begin{proof}
\
\begin{enumerate}[leftmargin=0.8cm, itemsep=5pt]
\item By Theorem \ref{JackLaplace} and the condition $\Re \, \alpha >\mu_0$, we obtain
\begin{align*}
&\int_{\R^n_+}\int_{\R^n} |g(x)|E^A(-\tfrac{1}{s},x^2+\underline{\epsilon})\Delta(s)^{-\Re \, \alpha-\mu_0-1}|\omega^B(x)|\omega^A(s) \dif x \dif s \\
&= \int_{\R^n} |g(x)| \int_{\R_+^n} E^A(-s,x^2+\underline{\epsilon})\Delta(s)^{\Re \,\alpha-\mu_0-1}\omega^A(s)\dif s \,  |\omega^B(x)| \dif x \\ 
&= \Gamma_n(\Re \,\alpha)\int_{\R^n} |g(x)|\Delta(x^2+\underline{\epsilon})^{-\Re \, \alpha} |\omega^B(x)| \dif x <\infty,
\end{align*}
since $g$ is a Schwartz function, and $x \mapsto \Delta(x^2+\epsilon)^{-\alpha} \omega^B(x)$ is continuous on $\R^n$, and of polynomial growth. Hence, it is justified to change the order of integration to obtain with Theorem \ref{JackLaplace} the stated equality.
\item Since $x \mapsto \Delta(x^2+\epsilon)^{- \alpha} \omega^B(x)$ is of polynomial growth, the integral exists for all $\alpha \in \C$. For $\Re \, \alpha < \Re \, \nu -\mu_0$ and $0<\epsilon<1$
$$\abs{\Delta(x^2+\underline{\epsilon})^{-\alpha}}\le \begin{cases}
\Delta(x^2)^{-\Re \, \alpha} & \te{ if } \Re \, \alpha >0 \\
\Delta(x^2+1)^{- \Re \, \alpha} & \te{ if } \Re \, \alpha \le 0
\end{cases}.$$
The dominated convergence theorem now proves the stated limit. Finally, analyticity follows by usual theorems on holomorphic parameter integrals.
\end{enumerate}
\end{proof}

\begin{lemma}\label{ZFLemma2}
Choose $m \in \N_0$ such that $\mu_0-m<\Re\, \nu -\mu_0$.
Consider $g \in \mathscr{S}(\R^n)$, $\tilde{g}(x)=g(x)\Delta(x)^m$ and let $\alpha \in \C$ with $\Re \, \alpha >\mu_0-m$.
\begin{enumerate}[leftmargin=0.8cm, itemsep=5pt]
\item[\rm{(i)}] For arbitrary $\epsilon>0$ the following integrals exist and coincide
\begin{align*}
\quad \quad \quad &\int_{\R_+^n}\int_{\R^n} g(x)E^A(-s,\tfrac{x^2}{4})E^A(-\tfrac{1}{s},\underline{\epsilon})\Delta(s)^{\nu-\alpha-\mu_0-1}\omega^B(x)\omega^A(s) \Delta(x)^m \dif x \dif s \\
&=\epsilon^{n(\nu-\alpha)}\int_{\R^n} g(x)\mathcal{K}_{\alpha-\nu}(\underline{1},\epsilon\tfrac{x^2}{4})\Delta(x)^m\omega^B(x)\dif x,
\end{align*}
where $\mathcal{K}_{\alpha-\nu}$ is the $\mathcal{K}$-Bessel function according to Definition \ref{DefKBessel}.
\item[\rm{(ii)}] The following limit exists
$$\quad \quad \quad  \lim\limits_{\epsilon \to 0}\epsilon^{n(\nu-\alpha)}\int_{\R^n} \tilde{g}(x)\mathcal{K}_{\alpha-\nu}(\underline{1},\epsilon\tfrac{x^2}{4})\omega^B(x)\dif x =\frac{\Gamma_n(\nu-\alpha)}{4^{n\alpha}}\mathcal{Z}(g,\alpha-\nu+m).$$
Moreover, the integral on the left hand side is holomorphic in $\alpha$ on the domain $\set{\Re \alpha >\mu_0-m}$.
\end{enumerate}
\end{lemma}

\begin{proof}
\
\begin{enumerate}[leftmargin=0.8cm, itemsep=5pt]
\item It suffices to justify, that we can change the order of integration. Everything else follows from the definition of the $\mathcal{K}$-Bessel function and its properties (Theorem \ref{KBessel}). The change of variables $s\mapsto \epsilon s$ and $\mathcal{K}_{\nu-\alpha}(w,z)=\mathcal{K}_{\alpha-\nu}(z,w)$ lead to
\begin{small}
\begin{align*}
\quad \quad \quad &\int_{\R_+^n}\int_{\R^n} |g(x)|E^A(-s,\tfrac{x^2}{4})E^A(-\tfrac{1}{s},\underline{\epsilon})\Delta(s)^{\Re(\nu-\alpha)-\mu_0-1}\Delta(x)^m|\omega^B(x)|\omega^A(s)\dif x \dif s \\
&= \epsilon^{n\Re(\nu-\alpha)}\int_{\R^n} \abs{g(x)} \mathcal{K}_{\Re(\nu-\alpha)}(\epsilon\tfrac{x^2}{4},\underline{1}) |\omega^B(x)|\dif x.
\end{align*}
\end{small}
To see that this integral is finite, we observe that by Theorem \ref{KBessel} (iii)
\begin{align*}
\quad \quad \quad  &\Delta(x^2)^m \mathcal{K}_{\Re(\alpha-\nu)}(\underline{1},\epsilon \tfrac{x^2}{4})   \\
&\le C(\alpha)\begin{cases}
1& \te{ if } \Re \, \alpha > \Re \, \nu +\mu_0 \\
1+\Delta(x^2)^{-\mu_0-\epsilon'} &\te{ if } \nu-\mu_0\le \Re \, \alpha \le \Re \, \nu+\mu_0 \\
& \te{ and arbitrary }0<\epsilon'<1 \\
\epsilon^{n(\Re\, \alpha-\Re \, \nu)}\Delta(x^2)^{\Re \, \alpha -\nu} & \te{ if }\Re \, \alpha < \Re, \nu-\mu_0
\end{cases},
\end{align*}
with some constant $C(\alpha)$. Hence, under the conditions $\Re \, \alpha >\mu_0-m$ and $\mu_0-m < \Re\, \nu-\mu_0$, the function $$x \mapsto \Delta(x^2)^m \mathcal{K}_{\Re(\alpha-\nu)}(\underline{1},\epsilon \tfrac{x^2}{4})$$ is of polynomial growth (choose $0<\epsilon'<m-2\mu_0-\Re \, \nu$ in the second case), and the integral is finite.
\item The same estimates as in part (i), dominated converges, and the asymptotics of Theorem \ref{KBessel} (vi) lead to the stated limit.
%\item By the same estimates as proven in part (i), we observe that the functions $x \mapsto \epsilon^{n(\nu-\alpha)}g(x)\mathcal{K}_{\alpha-\nu}(\underline{1},\epsilon\tfrac{x^2}{4})\Delta(x)^m\omega^B(x)$ have an uniform (in $\epsilon$) integrable bound. So it is justified to use dominated converges, and the asymptotics of Theorem \ref{KBessel} (vi) lead to the stated limit.  The holomorphicity is deduce from usual theorems on holomorphic parameter integrals.
\end{enumerate}
\end{proof}

\begin{proof}[Proof of Theorem \ref{FunctionalEquation}]
The left hand side of the functional equation is given by a positive measure if $\Re \,\alpha >\mu_0$, while the right hand side is a positive measure if $\Re\, \alpha<\Re\, \nu-\mu_0$. So there are possibly no indices $\alpha$ for which both sides are positive measures. To bypass this problem, we replace the argument $f \in \mathscr{S}(\R^n)$ by $\tilde{f}\coloneqq \Delta^{2m}(\Delta(T^B)^{2m}f) \in \mathscr{S}(\R^n)$ with large $m\in \N_0$, so large that $\mu_0-m<\Re \, \nu -\mu_0$.
\begin{enumerate}[leftmargin=0.8cm, itemsep=5pt]
\item Pick $\alpha \in \C$ with $\Re \, \alpha > \mu_0$. 
For $s \in \R_+^n$ put $e_s(x)=E^A(-s,x^2)$. The type $B$ Dunkl transform of $e_{1/s}$ is computed by Lemma \ref{HankelProperty} and Theorem \ref{HankelDunkl} (ii) as
$$\mathcal{F}^Be_{1/s}(x)=2^{-n\nu}e_s(\tfrac{x}{2})\Delta(s)^\nu.$$
The Plancherel theorem (cf. \cite{dJ93}) for the Dunkl transform leads to
$$\quad\quad  \int_{\R^n} (\mathcal{F}^B\tilde{f})(x)e_{1/s}(x)\omega^B(x) \dif x 
=2^{-n\nu}\Delta(s)^{\nu}\int_{\R^n} \tilde{f}(x)\m e_{s}(\tfrac{x}{2})\omega^B(x) \dif x.$$
Multiplying this equation with $\Delta(s)^{-\alpha-\mu_0-1}E^A(-\tfrac{1}{s},\underline{\epsilon})$, $\epsilon >0$ and integrating over $\R_+^n$ gives
\small\begin{equation*}
\begin{split}
\int_{\R^n_+}\int_{\R^n} (\mathcal{F}^B\tilde{f})(x)E^A(-\tfrac{1}{s},x^2+\underline{\epsilon})\Delta(s)^{-\alpha-\mu_0-1}\omega^B(x)\omega^A(s) \dif x \dif s \\
\quad \quad \quad =2^{-n\nu}\! \int_{\R_+^n}\int_{\R^n} \tilde{f}(x)E^A(-s,\tfrac{x^2}{4})E^A(-\tfrac{1}{s},\underline{\epsilon})\Delta(s)^{\nu-\alpha-\mu_0-1}\omega^B(x)\omega^A(s) \dif x \dif s.
\end{split}
\end{equation*}
By parts (i) of Lemmata \ref{ZFLemma1} and \ref{ZFLemma2}, this reduces to
\begin{align}\label{ZFE2}
\begin{split}
\int_{\R^n} \mathcal{F}^B \tilde{f}(x)&\Delta(x^2+\underline{\epsilon})^{-\alpha}\omega^B(x) \dif x \\
&= \frac{2^{-n\nu}\epsilon^{n(\nu-\alpha)}}{\Gamma_n(\alpha)}\int_{\R^n} \tilde{f}(x)\mathcal{K}_{\alpha-\nu}(\underline{1},\epsilon\tfrac{x^2}{4})\omega^B(x) \dif x.
\end{split}
\end{align}
In view of Lemma \ref{ZFLemma1} (ii), the left hand side of \eqref{ZFE2} exists for all $\alpha \in \C$ and is holomorphic in $\alpha$. For $\Re \, \alpha > \mu_0-m$ the right hand side of \eqref{ZFE2} exists and depends holomorphically on $\alpha$ by Lemma \ref{ZFLemma2} (ii), because 
$$\tilde{f}(x)\mathcal{K}_{\alpha-\nu}(\underline{1},\epsilon\tfrac{x^2}{4}) = \Delta(T^B)^{2m}f(x) \m \Delta(x^2)^m\mathcal{K}_{\alpha-\nu}(\underline{1},\epsilon\tfrac{x^2}{4})$$ and $\Delta(T^B)^{2m}f \in \mathscr{S}(\R^n)$. Hence, equality \eqref{ZFE2} is true for all $\Re \, \alpha >\mu_0-m$.
\item Assume that $\alpha \in \C$ is contained in the strip $\mu_0-m<\Re \,\alpha < \Re \, \nu-\mu_0$, which is non-empty by our choice of $m$. Then, on the one hand, by Lemma \ref{ZFLemma1}
\begin{align*}
&\int_{\R^n} \mathcal{F}^B \tilde{f}(x)\Delta(x^2+\underline{\epsilon})^{-\alpha}\omega^B(x) \dif x \\
&\underset{\epsilon \to 0}{\longrightarrow} \int_{\R^n} \mathcal{F}^B\tilde{f}(x) \Delta(x^2)^{-\alpha}\omega^B(x) \dif x \\
&= \mathcal{Z}(\mathcal{F}^B \tilde{f}; \nu-\alpha) \\
&= \Gamma_n(\nu-\alpha)\braket{\zeta_\alpha,\mathcal{F}^B\tilde{f}} \\
&= \Gamma_n(\nu-\alpha)\braket{\zeta_\alpha,\Delta(T^B)^{2m}(\Delta^{2m}f)} \\
&= \Gamma_n(\nu-\alpha)\; p(\alpha)\braket{\zeta_\alpha,f}, 
\end{align*}
with some polynomial $p$ (independent of $f$), where Theorem \ref{ZetaProperties} was used. \\
On the other hand, Lemma \ref{ZFLemma2} shows that
\begin{align*}
\quad \quad \frac{2^{-n\nu}\epsilon^{n(\nu-\alpha)}}{\Gamma_n(\alpha)}&\int_{\R^n} \tilde{f}(x)\mathcal{K}_{\alpha-\nu}(\underline{1},\epsilon\tfrac{x^2}{4})\omega^B(x) \dif x \\
&\underset{\epsilon \to 0}{\longrightarrow} \frac{\Gamma_n(\nu-\alpha)}{\Gamma(\alpha)2^{n(\nu+2\alpha}}\int_{\R^n} (\Delta(T^B)^{2m}f)(x)\m \Delta(x^2)^{\alpha-\nu+m}\omega^B(x) \dif x \\
&=\frac{\Gamma_n(\nu-\alpha)}{\Gamma_n(\alpha)2^{n(\nu+2\alpha)}}\mathcal{Z}(\Delta(T^B)^{2m}f;\alpha-\nu+m) \\
&= \frac{\Gamma_n(\nu-\alpha)}{\Gamma_n(\alpha)\Gamma_n(\alpha-\nu+m)2^{n(\nu+2\alpha)}} \braket{\zeta_{\alpha-\nu+m},\Delta(T^B)^{2m}f} \\
&= \frac{\Gamma_n(\nu-\alpha)}{\Gamma_n(\alpha)\Gamma_n(\alpha-\nu+m)}\; q(\alpha)\; \braket{\zeta_\alpha,f},
\end{align*}
with some polynomial $q$ (independent of $f$), where Theorem \ref{ZetaProperties} was used. 
Therefore, we have found a meromorphic function $\rho$ on $\C$ such that
$$\zeta_\alpha=\rho(\alpha) \mathcal{F}^B\zeta_{\nu-\alpha}.$$
Finally, by equation \eqref{GaussZeta} of Example \ref{GaussFunction}, we see that $\rho(\alpha)=2^{n(2\alpha-\nu)}$.
\end{enumerate} 
\end{proof}

\section{Regularity of the zeta distributions}
As for Riesz distributions in \cite{R20}, one may ask which of the zeta distributions are regular and which are positive measures. The case of zeta distributions reduces to the case of Riesz distributions, where the question was answered in \cite{R20}. It turns out that the set of indices $\alpha \in \C$ for which $\zeta_\alpha$ is a positive measure only depends on the multiplicity parameter $k$ of $\kappa_B=(k,k')$. 

\begin{lemma}
Consider $\alpha \in \C$.
\begin{enumerate}[leftmargin=0.8cm, itemsep=5pt]
\item[\rm{(i)}] $\zeta_\alpha$ is a positive or complex measure iff $R_\alpha$ is a positive or complex measure, respectively.
\item[\rm{(ii)}] If $\zeta_\alpha$ is a complex measure, then $\Re \, \alpha > \mu_0$ or $\alpha$ is contained in the finite set
$$[0,\infty[ \;\cap\; (\set{0,k,\ldots,k(n-1)}-\N_0).$$
\item[\rm{(iii)}] $\zeta_\alpha$ is a positive measure iff $\alpha$ is contained in the generalized Wallach set
$$W_k = \set{0,k,\ldots,k(n-1)}\cup \;]k(n-1),\infty[.$$
\end{enumerate}
\end{lemma}

\begin{proof}
If we can show part (i), then the other parts are immediate from \cite[Theorem 5.15]{R20}. First assume that $R_\alpha$ is a measure $\mu$. Then $\supp \, \mu \tm \overline{\R_+^n}$ (cf. \cite{R20}). Let $f \in \mathscr{S}(\R^n)$ and choose $f_0 \in \mathscr{S}(\R^n)$ so that
$$\frac{1}{2^n}\sum\limits_{\tau \in \Z_2^n} f(\tau x) = f_0(x^2).$$
Thus, the $W_{\! B}$-invariance of $\zeta_\alpha$ gives
\begin{align*}
\braket{\zeta_\alpha,f} &= \braket{R_\alpha,f_0}=\int_{\overline{\R_+^n}} f_0(x)  \dif \mu(x) =\frac{1}{2^n}\sum\limits_{\tau \in \Z_2^n} \int_{\overline{\R_+^n}} f(\tau \sqrt{x} \dif \mu(x,
\end{align*}
i.e. $\zeta_\alpha$ is a measure. Conversely, if $\zeta_\alpha$ is a measure $\mu$, then for $f \in \mathscr{S}(\R^n)$ 
$$\braket{R_\alpha,f}=\braket{\zeta_\alpha, f(x^2)}=\int_{\R^n} f(x^2)\; \dif \mu(x),$$
i.e. $R_\alpha$ is a measure.
\end{proof}

To conclude the study of zeta distributions, we shall explicitly compute $\zeta_\alpha$ for $\alpha \in \set{0,k,\ldots,k(n-1)}$, the discrete part of the generalized Wallach set. On the continuous part of the generalized Wallach set, $\zeta_\alpha$ is given by the measure 
$$\tfrac{1}{\Gamma_n(\alpha)}\Delta(x^2)^{\alpha-\nu}\omega^B(x) \dif x=\tfrac{1}{\Gamma_n(\alpha)}\Delta(x^2)^\alpha\omega^A(x^2)\dif x.$$
In the following, $W_{\! B}$ acts on $\mathscr{S}'(\R^n)$ by $\braket{w u,f}=\braket{u,w^{-1}f}$ for $u \in \mathscr{S}'(\R^n),\, f \in \mathscr{S}(\R^n)$ and $w \in W_{\! B}$.

\begin{theorem}
The zeta distribution on the discrete part of the Wallach set are given as
\begin{enumerate}[leftmargin=0.8cm, itemsep=5pt]
\item[\rm{(i)}] $\zeta_0=\delta_0$.
\item[\rm{(ii)}] For $r=1,\ldots,n-1$, $\zeta_{kr}$ is a positive measure, namely
$$\zeta_{kr}=\frac{1}{n!}\sum\limits_{\sigma \in \mathcal{S}_n} (\zeta_{kn}^{(r)}\otimes \delta_0^{(n-r)})^\sigma,$$
where $\zeta_{kn}^{(r)}$ is the zeta distribution of index $kn$ on $\R^r$ and $\delta_0^{(n-r)}$ is the Dirac measure in $0 \in \R^{n-r}$. Moreover, the support of $\zeta_{kr}$ is the stratum
$$ \quad \partial_r\R^n\coloneqq \set{x\in \R^n \mid x_{i_j}= 0 \te{ for a sequence } 1\le i_1 < \ldots < i_r \le n}.$$
\end{enumerate}
\end{theorem}

\begin{proof}
\
\begin{enumerate}[leftmargin=0.8cm, itemsep=5pt]
\item By Theorem \ref{FunctionalEquation} and Theorem \ref{HankelDunkl} (i) we obtain
\begin{align*}
\braket{\zeta_0,f}&=2^{-n\nu}\braket{\zeta_\nu,\mathcal{F}^Bf}=\frac{1}{2^{n\nu}\Gamma_n(\nu)}\int_{\R^n} \mathcal{F}^Bf(x) \omega^B(x) \dif x \\
&=\mathcal{F}^B\mathcal{F}^Bf(0)=f(0)=\braket{\delta_0,f}.
\end{align*}
\item It suffices to verify the stated formula for $W_{\! B}$-invariant $f \in \mathscr{S}(\R^n)$. Hence there is an $\mathcal{S}_n$-invariant $f_0 \in \mathscr{S}(\R^n)$ with $f(x)=f_0(x^2)$. By \cite[Theorem 5.11]{R20} we have
$$R_{kr}=\frac{1}{n!}\sum\limits_{\sigma \in \mathcal{S}_n}(R_{kn}^{(r)}\otimes \delta_0^{(n-r)})^\sigma,$$
where $R_{kn}^{(r)}$ is the Dunkl-type Riesz distribution on $\R^r$. Therefore
\begin{align*}
\braket{\zeta_{kr},f}&=\braket{R_{kr},f_0}=\frac{1}{n!}\sum\limits_{\sigma \in \mathcal{S}_n} \braket{(R_{kr}^{(r)}\otimes \delta_0^{(n-r)})^{\sigma},f_0} \\
&= \braket{R_{kr}^{(r)},f_0(\m,0^{n-r})}=\braket{\zeta_{kr}^{(r)},f(\m,0^{n-r})}\\
&= \frac{1}{n!}\sum\limits_{\sigma \in \mathcal{S}_n} \braket{(\zeta_{kr}^{(r)}\otimes \delta_0^{(n-r)})^{\sigma},f}.
\end{align*}
Since $r<n$, the index $kn$ is contained in the continuous part of the Wallach set associated to the zeta distributions $\zeta_\alpha^{(r)}$ on $\R^r$. Thus, $\zeta_{kn}^{(r)}$ has support $\R^r$, by Theorem \ref{ZetaProperties} (ii). In particular, the support of $\zeta_{kr}$ is 
$$\bigcup\limits_{\sigma \in \mathcal{S}_n} \sigma(\R^r\times \set{0}^{n-r}) = \partial_r\R^n.$$
\end{enumerate}
\end{proof}

\section*{Acknowledgments}
This work was funded by the Deutsche Forschungsgemeinschaft [RO 1264/4-1].

\end{document}

%% file: MyMacros.tex
\newcommand{\C}{\mathbb{C}}
\newcommand{\R}{\mathbb{R}}

\newcommand{\Z}{\mathbb{Z}}

\newcommand{\N}{\mathbb{N}}

\newcommand{\tm}{\subseteq}

\newcommand{\supp}{\mathrm{supp}}

\renewcommand{\Re}{\mathrm{Re}}

\renewcommand{\d}{\mathrm{d}}
\newcommand{\dif}{\; \mathrm{d}}
\newcommand{\spec}{\mathrm{spec}}

\newcommand{\tr}{\mathrm{tr}}

\newcommand{\nrm}[1]{\left\lVert#1\right\rVert}
\newcommand{\abs}[1]{\ensuremath{\left\vert#1\right\vert}}

\newcommand{\te}{\textrm}
\newcommand{\m}{\cdot} 
\renewcommand{\tilde}[1]{\widetilde{#1}}

\renewcommand{\set}[1]{\left\{ #1 \right\}}

\renewcommand\labelenumi{(\roman{enumi})}
\renewcommand\theenumi\labelenumi
\makeindex